\documentclass[amssymb,12pt]{article}
\usepackage[dvips]{graphicx}
\usepackage{amscd,amssymb,amsopn,amsmath,amsthm,graphics,amsfonts,enumerate,verbatim,calc}
\usepackage{graphicx}
\usepackage{enumerate}
\usepackage{mathtools}
\usepackage{indentfirst}
\usepackage{pstricks}
\usepackage{epsfig}
\usepackage{pgfplots}
\pgfplotsset{compat=1.14}
\newtheorem{theorem}{Theorem}

\newtheorem{corollary}{Corollary}
\newtheorem{remark}{Remark}

\newtheorem{definition}{Definition}
\begin{document}

\title{The shape of the reliability polynomial\\ of a hammock network}

\author{Leonard  D\u au\c s$^{1}$, \; Marilena Jianu$^1$}
\date{} 
\maketitle
\begin{center}{\small
$^1$ Department of Mathematics and Computer Science, \\
Technical University of Civil Engineering of Bucharest, \\ 020396, Bd. Lacul Tei, 124, Bucharest, Romania
}\end{center}

\begin{abstract}
Motivated by the study of hammock (aka brick-wall) networks, we introduce in this paper the notion of X-path. Using the Jordan Curve Theorem for piecewise smooth curves, we prove duality properties for hammock networks. Consequences for reliability polynomials are given.  
\end{abstract}



\section{Introduction}

The concept of network reliability can be traced back to 1956, when Von Neumann \cite{N}, and Moore and Shannon \cite{MS}, respectively, published two prescient papers. The original purpose of Moore and Shannon was to understand the reliability of electrical circuits/networks  having unreliable individual components. In order to improve the reliability of such networks, they introduced a new type of reliability enhancement scheme called brick-wall (or hammock) network. In the last few years, the interest regarding the work of Moore and Shannon has been growing (see \cite{CDBP}, \cite{DCBHG}, \cite{NRKE}, \cite{NRE}), not only from theoretical point of view, but also because of its applicability in various fields ranging from biology/medicine to engineering and even social sciences.

The problem of finding the reliability polynomial of a network belongs to the class of $\#P-$complete problems, a class of computationally equivalent counting problems (introduced by Valiant in \cite{V1}) that are at least as difficult as the $NP-$complete problems (\cite{V2}, 
\cite{Co}). Although the brick-wall networks were proposed more than sixty years ago, for their reliability polynomials no general close form formula have been reported yet. Recently, in \cite{CDBP}, the reliability polynomials have been calculated exactly for a few particular cases of small size, more precisely for the 29 hammock networks presented by Moore and Shannon in their original paper. For completing this task, the authors used an algorithm based on a recursive depth-first traversal of a binary tree. Another important step was achieved in \cite{DBCP} where the first and second non-zero coefficients of the reliability polynomial have been computed, for any hammock network. The methods used to prove the formulas for these leading coefficients involve the transition matrix of certain linear transformations, lattice paths and generating functions.

The main goal of this paper is to propose, in Theorem \ref{t1}, a direct proof of duality properties for hammock networks. An important consequence is a significant reduction of requested calculus for finding reliability polynomials of all hammock networks. It should be noted that, while this paper studies an applied mathematical subject, it uses, as a key tool, the Jordan curve theorem, which is a pure mathematical result (in fact, it is the first theorem discovered in set-theoretic topology). 

The concept of brick-wall lattice path, introduced in \cite{DBCP}, has been proved to be a useful and versatile tool in the study of brick-wall networks. Here, when studying hammock networks, it is natural to define and to use the concept of $X-$path, which generates all possible connections through the network.  We refer the reader to \cite{Hu10} and \cite{St86} for more details about lattice paths, and to \cite{Co} for definitions and results about network reliability.

\section{The reliability polynomial of a network}

A network is a probabilistic graph \cite{Co}, $N=\left(V,E\right)$, where \textit{V} is the set of nodes (vertices) and \textit{E } is the set of (undirected) edges. The edges can be represented as independent identically distributed random variables: each edge operates (is closed) with probability \textit{p} and fails (is open) with probability $q=1-p$. We assume that nodes do not fail, hence the failure of the network is always a consequence of edge failures.

Let $K$ be a subset of $V$ containing some special nodes (called terminals). The \textit{K -} reliability of the network \textit{N} is the probability that there exists a path (a sequence of adjacent edges) made of operational (closed) edges between any pair of nodes in \textit{K}. This is a polynomial function in \textit{p} denoted by $h_{K} (p)$. If $K=V$ then $h_{K} (p)$ is called the all-terminal  reliability of the network. If the subset \textit{K} contains two nodes: \textit{S} (source/input) and \textit{T }(terminus/output) then $h_{K} (p)$ is called two-terminal   reliability. This paper studies the two-terminal reliability only and we denote it by $h(p)$.

A \textit{pathset} in the network \textit{N} is a subset of \textit{E }which contains a path between the nodes\textit{ S} and \textit{T}. A minimal pathset (\textit{minpath}) is a pathset \textit{P} such that, if any edge \textit{e }of \textit{P} is removed, then $P-\left\{e\right\}$ is no longer a pathset (the nodes \textit{S} and \textit{T} are disconnected). We denote by $\mathcal{P}$ the set of all the pathsets of \textit{N. }

A \textit{cutset} in the network \textit{N} is a subset of edges, $C\subset E$, such that the complementary set, $E-C$, contains no path between \textit{S} and \textit{T }($E-C$ is not a pathset). A minimal cutset (\textit{mincut}) is a cutset \textit{C} such that, if any edge \textit{e }of \textit{C} is removed, then $C-\left\{e\right\}$ is no longer a cutset ($E-C\bigcup \left\{e\right\}$ is a pathset). We denote by ${\mathcal C}$ the set of all the cutsets of \textit{N. }

If $n=\left|E\right|$ is the size of the graph, $N_{i} $ is the number of pathsets with exactly \textit{i }edges and  $C_{i} $, the number of cutsets with exactly \textit{i }edges, then the reliability of the network can be expressed as (see [1])
\begin{equation}\label{1}
 h(p)=\sum _{P\in \mathcal{P}}p^{|P|} q^{n-|P|}  =\sum _{i=1}^{n}N_{i} p^{i} (1-p)^{n-i},
\end{equation}
or, in terms of cutsets, as
\begin{equation}\label{2}
 h(p)=1-\sum _{C\in  {\mathcal C}}q^{|C|} p^{n-|C|}  =1-\sum _{i=1}^{n}C_{i} (1-p)^{i} p^{n-i} .
\end{equation}

\section{Hammock networks}

A brick-wall network is formed by $w\times l$ identical devices disposed in \textit{w} lines, each line consisting of \textit{l} devices connected in series. Besides the horizontal connections, there exist also vertical connections. Out of all $(l-1)(w-1)$ possible vertical connections, half are present and the other half are absent. The vertical connections are arranged regularly in an alternate way which gives rise to the ``brick-wall'' pattern shown in Fig. \ref{Fig 1}.

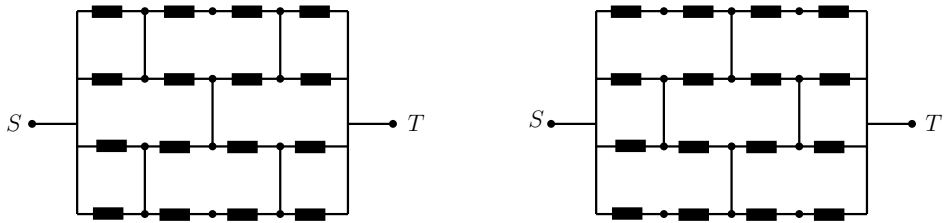
\begin{figure}[h]
\centering
\psscalebox{0.75} 
{
\begin{pspicture}(0,-1.91)(16.53,1.91)
\psline[linecolor=black, linewidth=0.04](1.26,-1.79)(6.06,-1.79)
\psline[linecolor=black, linewidth=0.04](1.26,-0.59)(6.06,-0.59)
\psline[linecolor=black, linewidth=0.04](1.26,0.61)(6.06,0.61)
\psline[linecolor=black, linewidth=0.04](1.26,-1.79)(1.26,1.81)
\psline[linecolor=black, linewidth=0.04](6.06,-1.79)(6.06,1.81)
\psline[linecolor=black, linewidth=0.04](1.26,-0.19)(0.46,-0.19)
\psline[linecolor=black, linewidth=0.04](6.06,-0.19)(6.86,-0.19)
\psdots[linecolor=black, dotsize=0.14](6.86,-0.19)
\psdots[linecolor=black, dotsize=0.14](0.46,-0.19)
\psframe[linecolor=black, linewidth=0.04, fillstyle=solid,fillcolor=black, dimen=outer](2.06,0.71)(1.52,0.49)
\psframe[linecolor=black, linewidth=0.04, fillstyle=solid,fillcolor=black, dimen=outer](2.08,-1.67)(1.54,-1.89)
\psframe[linecolor=black, linewidth=0.04, fillstyle=solid,fillcolor=black, dimen=outer](2.14,-0.47)(1.6,-0.69)
\psframe[linecolor=black, linewidth=0.04, fillstyle=solid,fillcolor=black, dimen=outer](5.76,0.71)(5.22,0.49)
\psframe[linecolor=black, linewidth=0.04, fillstyle=solid,fillcolor=black, dimen=outer](4.54,0.71)(4.0,0.49)
\psframe[linecolor=black, linewidth=0.04, fillstyle=solid,fillcolor=black, dimen=outer](3.34,0.71)(2.8,0.49)
\psdots[linecolor=black, dotsize=0.14](2.46,-0.59)
\psdots[linecolor=black, dotsize=0.14](2.46,0.61)
\psdots[linecolor=black, dotsize=0.14](3.66,-0.59)
\psdots[linecolor=black, dotsize=0.14](3.66,-1.79)
\psdots[linecolor=black, dotsize=0.14](4.86,-0.59)
\psdots[linecolor=black, dotsize=0.14](4.86,0.61)
\psdots[linecolor=black, dotsize=0.14](3.66,0.61)
\psdots[linecolor=black, dotsize=0.14](2.46,-1.79)
\psdots[linecolor=black, dotsize=0.14](4.86,-1.79)
\psframe[linecolor=black, linewidth=0.04, fillstyle=solid,fillcolor=black, dimen=outer](5.66,-0.49)(5.12,-0.71)
\psframe[linecolor=black, linewidth=0.04, fillstyle=solid,fillcolor=black, dimen=outer](3.26,-0.49)(2.72,-0.71)
\psframe[linecolor=black, linewidth=0.04, fillstyle=solid,fillcolor=black, dimen=outer](4.46,-0.49)(3.92,-0.71)
\psframe[linecolor=black, linewidth=0.04, fillstyle=solid,fillcolor=black, dimen=outer](3.26,-1.69)(2.72,-1.91)
\psframe[linecolor=black, linewidth=0.04, fillstyle=solid,fillcolor=black, dimen=outer](4.46,-1.69)(3.92,-1.91)
\psframe[linecolor=black, linewidth=0.04, fillstyle=solid,fillcolor=black, dimen=outer](5.66,-1.69)(5.12,-1.91)
\psline[linecolor=black, linewidth=0.04](1.26,1.81)(6.06,1.81)
\psframe[linecolor=black, linewidth=0.04, fillstyle=solid,fillcolor=black, dimen=outer](2.06,1.91)(1.52,1.69)
\psframe[linecolor=black, linewidth=0.04, fillstyle=solid,fillcolor=black, dimen=outer](4.54,1.91)(4.0,1.69)
\psframe[linecolor=black, linewidth=0.04, fillstyle=solid,fillcolor=black, dimen=outer](5.74,1.91)(5.2,1.69)
\psframe[linecolor=black, linewidth=0.04, fillstyle=solid,fillcolor=black, dimen=outer](3.34,1.91)(2.8,1.69)
\psdots[linecolor=black, dotsize=0.14](2.46,1.81)
\psdots[linecolor=black, dotsize=0.14](3.66,1.81)
\psdots[linecolor=black, dotsize=0.14](4.86,1.81)
\psline[linecolor=black, linewidth=0.04](2.46,-0.59)(2.46,-1.79)
\psline[linecolor=black, linewidth=0.04](4.86,-0.59)(4.86,-1.79)
\psline[linecolor=black, linewidth=0.04](3.66,0.61)(3.66,-0.59)
\psline[linecolor=black, linewidth=0.04](2.46,1.81)(2.46,0.61)
\psline[linecolor=black, linewidth=0.04](4.86,1.81)(4.86,0.61)
\rput[bl](0.0,-0.31){$S$}
\rput[bl](7.12,-0.31){$T$}
\psline[linecolor=black, linewidth=0.04](10.46,-1.79)(10.46,1.81)
\psline[linecolor=black, linewidth=0.04](10.46,1.81)(15.26,1.81)
\psframe[linecolor=black, linewidth=0.04, fillstyle=solid,fillcolor=black, dimen=outer](11.26,1.91)(10.72,1.69)
\psdots[linecolor=black, dotsize=0.14](11.66,1.81)
\psline[linecolor=black, linewidth=0.04](12.86,1.81)(12.86,0.61)
\psline[linecolor=black, linewidth=0.04](10.46,-1.79)(15.26,-1.79)
\psline[linecolor=black, linewidth=0.04](10.46,-0.59)(15.26,-0.59)
\psline[linecolor=black, linewidth=0.04](10.46,0.61)(15.26,0.61)
\psline[linecolor=black, linewidth=0.04](15.26,-1.79)(15.26,1.81)
\psline[linecolor=black, linewidth=0.04](10.46,-0.19)(9.66,-0.19)
\psline[linecolor=black, linewidth=0.04](15.26,-0.19)(16.06,-0.19)
\psdots[linecolor=black, dotsize=0.14](16.06,-0.19)
\psdots[linecolor=black, dotsize=0.14](9.66,-0.19)
\psframe[linecolor=black, linewidth=0.04, fillstyle=solid,fillcolor=black, dimen=outer](11.26,0.71)(10.72,0.49)
\psframe[linecolor=black, linewidth=0.04, fillstyle=solid,fillcolor=black, dimen=outer](11.28,-1.67)(10.74,-1.89)
\psframe[linecolor=black, linewidth=0.04, fillstyle=solid,fillcolor=black, dimen=outer](11.34,-0.47)(10.8,-0.69)
\psframe[linecolor=black, linewidth=0.04, fillstyle=solid,fillcolor=black, dimen=outer](14.96,0.71)(14.42,0.49)
\psframe[linecolor=black, linewidth=0.04, fillstyle=solid,fillcolor=black, dimen=outer](13.74,0.71)(13.2,0.49)
\psframe[linecolor=black, linewidth=0.04, fillstyle=solid,fillcolor=black, dimen=outer](12.54,0.71)(12.0,0.49)
\psdots[linecolor=black, dotsize=0.14](11.66,-0.59)
\psdots[linecolor=black, dotsize=0.14](11.66,0.61)
\psdots[linecolor=black, dotsize=0.14](12.86,-0.59)
\psdots[linecolor=black, dotsize=0.14](12.86,-1.79)
\psdots[linecolor=black, dotsize=0.14](14.06,-0.59)
\psdots[linecolor=black, dotsize=0.14](14.06,0.61)
\psdots[linecolor=black, dotsize=0.14](12.86,0.61)
\psdots[linecolor=black, dotsize=0.14](11.66,-1.79)
\psdots[linecolor=black, dotsize=0.14](14.06,-1.79)
\psframe[linecolor=black, linewidth=0.04, fillstyle=solid,fillcolor=black, dimen=outer](14.86,-0.49)(14.32,-0.71)
\psframe[linecolor=black, linewidth=0.04, fillstyle=solid,fillcolor=black, dimen=outer](12.46,-0.49)(11.92,-0.71)
\psframe[linecolor=black, linewidth=0.04, fillstyle=solid,fillcolor=black, dimen=outer](13.66,-0.49)(13.12,-0.71)
\psframe[linecolor=black, linewidth=0.04, fillstyle=solid,fillcolor=black, dimen=outer](12.46,-1.69)(11.92,-1.91)
\psframe[linecolor=black, linewidth=0.04, fillstyle=solid,fillcolor=black, dimen=outer](13.66,-1.69)(13.12,-1.91)
\psframe[linecolor=black, linewidth=0.04, fillstyle=solid,fillcolor=black, dimen=outer](14.86,-1.69)(14.32,-1.91)
\psframe[linecolor=black, linewidth=0.04, fillstyle=solid,fillcolor=black, dimen=outer](13.74,1.91)(13.2,1.69)
\psframe[linecolor=black, linewidth=0.04, fillstyle=solid,fillcolor=black, dimen=outer](14.94,1.91)(14.4,1.69)
\psframe[linecolor=black, linewidth=0.04, fillstyle=solid,fillcolor=black, dimen=outer](12.54,1.91)(12.0,1.69)
\psdots[linecolor=black, dotsize=0.14](12.86,1.81)
\psdots[linecolor=black, dotsize=0.14](14.06,1.81)
\psline[linecolor=black, linewidth=0.04](11.66,0.61)(11.66,-0.59)
\psline[linecolor=black, linewidth=0.04](14.06,0.61)(14.06,-0.59)
\psline[linecolor=black, linewidth=0.04](12.86,-0.59)(12.86,-1.79)
\rput[bl](9.26,-0.25){$S$}
\rput[bl](16.3,-0.31){$T$}
\end{pspicture}
}
	\caption{Brick-wall networks of dimensions $w=4$, $l=4$}
	\label{Fig 1}
\end{figure}


Brick-wall networks were also named hammock networks by Moore and Shannon, from their appearance when the nodes \textit{S} and \textit{T} are pulled apart and every vertical connection collapses into a node. In this case, rectangular ``bricks'' deform into rhombs. As can be seen from Fig. \ref{Fig 2}, the ``hammock'' representation fits the above definition of the probabilistic graph, unlike the ``brick-wall'' representation, where the vertical edges have no probability assigned to them (i.e., it is assumed they are always closed). 

In Fig. \ref{Fig 2} a brick-wall network with $w=3$, $l=7$ and the equivalent hammock network are presented. Notice that, in order to preserve the regularity of the hammock network, the nodes \textit{S} and \textit{T} can be replaced by ``fictive'' nodes, $S_{1} ,S_{2} ,\ldots,S_k $, and, respectively, $T_{1} ,T_{2} ,\ldots,T_h $, where $k,h\in\left\{\left\lfloor \frac{w}{2}\right\rfloor, \left\lfloor \frac{w}{2}\right\rfloor+1\right\}$.

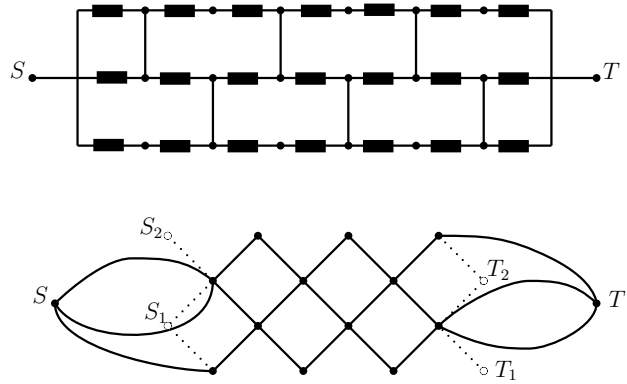
\begin{figure}[h]\centering
\psscalebox{0.75} 
{
\begin{pspicture}(0,-3.34)(10.87,3.34)
\psline[linecolor=black, linewidth=0.04, linestyle=dotted, dotsep=0.10583334cm](3.6,-1.58)(2.8,-0.78)
\psline[linecolor=black, linewidth=0.04, linestyle=dotted, dotsep=0.10583334cm](3.6,-1.58)(2.8,-2.38)(3.6,-3.18)
\psline[linecolor=black, linewidth=0.04, linestyle=dotted, dotsep=0.10583334cm](7.6,-0.78)(8.4,-1.58)(7.6,-2.38)(8.4,-3.18)
\psbezier[linecolor=black, linewidth=0.04](0.8,-1.98)(0.8,-1.98)(1.6,-1.18)(2.16,-1.18)(2.72,-1.18)(3.12,-1.18)(3.6,-1.58)
\psbezier[linecolor=black, linewidth=0.04](0.8,-1.98)(1.2666667,-2.78)(3.6,-2.78)(3.6,-1.58)
\psbezier[linecolor=black, linewidth=0.04](0.8,-1.98)(0.8,-2.78)(2.8,-3.18)(3.6,-3.18)
\psline[linecolor=black, linewidth=0.04](3.6,-1.58)(4.4,-0.78)(5.2,-1.58)(6.0,-0.78)(6.8,-1.58)(7.6,-0.78)
\psline[linecolor=black, linewidth=0.04](3.6,-1.58)(4.4,-2.38)(5.2,-1.58)(6.0,-2.38)(6.8,-1.58)(7.6,-2.38)
\psline[linecolor=black, linewidth=0.04](3.6,-3.18)(4.4,-2.38)(5.2,-3.18)(6.0,-2.38)(6.8,-3.18)(7.6,-2.38)
\psdots[linecolor=black, dotsize=0.14](0.8,-1.98)
\psdots[linecolor=black, dotsize=0.14](3.6,-1.58)
\psdots[linecolor=black, dotsize=0.14](4.4,-0.78)
\psdots[linecolor=black, dotsize=0.14](5.2,-1.58)
\psdots[linecolor=black, dotsize=0.14](4.4,-2.38)
\psdots[linecolor=black, dotsize=0.14](3.6,-3.18)
\psdots[linecolor=black, dotsize=0.14](5.2,-3.18)
\psdots[linecolor=black, dotsize=0.14](6.0,-0.78)
\psdots[linecolor=black, dotsize=0.14](6.0,-2.38)
\psdots[linecolor=black, dotsize=0.14](6.8,-1.58)
\psdots[linecolor=black, dotsize=0.14](6.8,-3.18)
\psdots[linecolor=black, dotsize=0.14](7.6,-2.38)
\psdots[linecolor=black, dotsize=0.14](7.6,-0.78)
\psdots[linecolor=black, fillstyle=solid, dotstyle=o, dotsize=0.14, fillcolor=white](2.8,-0.78)
\psdots[linecolor=black, fillstyle=solid, dotstyle=o, dotsize=0.14, fillcolor=white](2.8,-2.38)
\psdots[linecolor=black, fillstyle=solid, dotstyle=o, dotsize=0.14, fillcolor=white](8.4,-1.58)
\psdots[linecolor=black, fillstyle=solid, dotstyle=o, dotsize=0.14, fillcolor=white](8.4,-3.18)
\psbezier[linecolor=black, linewidth=0.04](7.6,-0.78)(8.4,-0.78)(8.8,-0.78)(9.6,-1.18)(10.4,-1.58)(10.4,-1.98)(10.4,-1.98)
\psbezier[linecolor=black, linewidth=0.04](7.6,-2.38)(8.4,-2.78)(8.8,-2.78)(9.2,-2.78)(9.6,-2.78)(10.4,-2.38)(10.4,-1.98)
\psbezier[linecolor=black, linewidth=0.04](7.6,-2.38)(8.0,-1.98)(8.8,-1.58)(9.2,-1.58)(9.6,-1.58)(10.0,-1.58)(10.4,-1.98)
\psdots[linecolor=black, dotsize=0.14](10.4,-1.98)
\rput[bl](0.4,-1.98){$S$}
\rput[bl](10.62,-2.04){$T$}
\rput[bl](2.3,-0.82){$S_2$}
\rput[bl](2.38,-2.32){$S_1$}
\rput[bl](8.46,-1.56){$T_2$}
\rput[bl](8.6,-3.34){$T_1$}
\psline[linecolor=black, linewidth=0.04](1.2,0.82)(9.6,0.82)
\psline[linecolor=black, linewidth=0.04](1.2,2.02)(9.6,2.02)
\psline[linecolor=black, linewidth=0.04](1.2,3.22)(9.6,3.22)
\psline[linecolor=black, linewidth=0.04](1.2,0.82)(1.2,3.22)
\psline[linecolor=black, linewidth=0.04](9.6,0.82)(9.6,3.22)
\psline[linecolor=black, linewidth=0.04](1.2,2.02)(0.4,2.02)
\psline[linecolor=black, linewidth=0.04](9.6,2.02)(10.4,2.02)
\psline[linecolor=black, linewidth=0.04](2.4,3.22)(2.4,2.02)
\psline[linecolor=black, linewidth=0.04](3.6,2.02)(3.6,0.82)
\psline[linecolor=black, linewidth=0.04](4.8,2.02)(4.8,3.22)
\psline[linecolor=black, linewidth=0.04](6.0,2.02)(6.0,0.82)
\psline[linecolor=black, linewidth=0.04](7.2,2.02)(7.2,3.22)
\psline[linecolor=black, linewidth=0.04](8.4,2.02)(8.4,0.82)
\psdots[linecolor=black, dotsize=0.14](10.4,2.02)
\psdots[linecolor=black, dotsize=0.14](0.4,2.02)
\psframe[linecolor=black, linewidth=0.04, fillstyle=solid,fillcolor=black, dimen=outer](2.0,3.32)(1.46,3.1)
\psframe[linecolor=black, linewidth=0.04, fillstyle=solid,fillcolor=black, dimen=outer](2.02,0.94)(1.48,0.72)
\psframe[linecolor=black, linewidth=0.04, fillstyle=solid,fillcolor=black, dimen=outer](2.08,2.14)(1.54,1.92)
\psframe[linecolor=black, linewidth=0.04, fillstyle=solid,fillcolor=black, dimen=outer](6.82,3.34)(6.28,3.12)
\psframe[linecolor=black, linewidth=0.04, fillstyle=solid,fillcolor=black, dimen=outer](5.7,3.32)(5.16,3.1)
\psframe[linecolor=black, linewidth=0.04, fillstyle=solid,fillcolor=black, dimen=outer](4.48,3.32)(3.94,3.1)
\psframe[linecolor=black, linewidth=0.04, fillstyle=solid,fillcolor=black, dimen=outer](3.28,3.32)(2.74,3.1)
\psdots[linecolor=black, dotsize=0.14](2.4,2.02)
\psdots[linecolor=black, dotsize=0.14](2.4,3.22)
\psdots[linecolor=black, dotsize=0.14](3.6,2.02)
\psdots[linecolor=black, dotsize=0.14](3.6,0.82)
\psdots[linecolor=black, dotsize=0.14](4.8,2.02)
\psdots[linecolor=black, dotsize=0.14](4.8,3.22)
\psdots[linecolor=black, dotsize=0.14](6.0,0.82)
\psdots[linecolor=black, dotsize=0.14](6.0,2.02)
\psdots[linecolor=black, dotsize=0.14](3.6,3.22)
\psdots[linecolor=black, dotsize=0.14](6.0,3.22)
\psdots[linecolor=black, dotsize=0.14](2.4,0.82)
\psdots[linecolor=black, dotsize=0.14](4.8,0.82)
\psdots[linecolor=black, dotsize=0.14](7.2,0.82)
\psdots[linecolor=black, dotsize=0.14](8.4,3.22)
\psdots[linecolor=black, dotsize=0.14](7.2,3.22)
\psdots[linecolor=black, dotsize=0.14](7.2,2.02)
\psdots[linecolor=black, dotsize=0.14](8.4,2.02)
\psdots[linecolor=black, dotsize=0.14](8.4,0.82)
\psframe[linecolor=black, linewidth=0.04, fillstyle=solid,fillcolor=black, dimen=outer](6.8,2.12)(6.26,1.9)
\psframe[linecolor=black, linewidth=0.04, fillstyle=solid,fillcolor=black, dimen=outer](5.6,2.12)(5.06,1.9)
\psframe[linecolor=black, linewidth=0.04, fillstyle=solid,fillcolor=black, dimen=outer](3.2,2.12)(2.66,1.9)
\psframe[linecolor=black, linewidth=0.04, fillstyle=solid,fillcolor=black, dimen=outer](4.4,2.12)(3.86,1.9)
\psframe[linecolor=black, linewidth=0.04, fillstyle=solid,fillcolor=black, dimen=outer](3.2,0.92)(2.66,0.7)
\psframe[linecolor=black, linewidth=0.04, fillstyle=solid,fillcolor=black, dimen=outer](4.4,0.92)(3.86,0.7)
\psframe[linecolor=black, linewidth=0.04, fillstyle=solid,fillcolor=black, dimen=outer](8.0,3.32)(7.46,3.1)
\psframe[linecolor=black, linewidth=0.04, fillstyle=solid,fillcolor=black, dimen=outer](5.6,0.92)(5.06,0.7)
\psframe[linecolor=black, linewidth=0.04, fillstyle=solid,fillcolor=black, dimen=outer](6.8,0.92)(6.26,0.7)
\psframe[linecolor=black, linewidth=0.04, fillstyle=solid,fillcolor=black, dimen=outer](8.0,0.92)(7.46,0.7)
\psframe[linecolor=black, linewidth=0.04, fillstyle=solid,fillcolor=black, dimen=outer](8.0,2.12)(7.46,1.9)
\psframe[linecolor=black, linewidth=0.04, fillstyle=solid,fillcolor=black, dimen=outer](9.2,0.92)(8.66,0.7)
\psframe[linecolor=black, linewidth=0.04, fillstyle=solid,fillcolor=black, dimen=outer](9.2,2.12)(8.66,1.9)
\psframe[linecolor=black, linewidth=0.04, fillstyle=solid,fillcolor=black, dimen=outer](9.2,3.32)(8.66,3.1)
\rput[bl](0.0,2.02){$S$}
\rput[bl](10.52,2.0){$T$}
\end{pspicture}
}
\caption{Hammock network of dimensions $w=3$, $l=7$}
	\label{Fig 2}
\end{figure}

\begin{definition} Let $S\subset {\bf {\mathbb Z}}^{2} $. A lattice path with steps in \textit{S} is a sequence of lattice points, $v_{0} ,v_{1} ,\ldots ,v_{k} \in {\bf {\mathbb Z}}^{2} $, such that $v_{i} -v_{i-1} \in S$ for all $i=0,1,\ldots ,k$.
\end{definition}

\begin{figure}[h]\centering
\psscalebox{0.75} 
{
\begin{pspicture}(0,-3.3534791)(7.975949,3.3534791)
\psline[linecolor=black, linewidth=0.04, arrowsize=0.05291667cm 2.0,arrowlength=1.4,arrowinset=0.0]{->}(0.06899033,-1.7534791)(8.068991,-1.7534791)
\psline[linecolor=black, linewidth=0.04, arrowsize=0.05291667cm 2.0,arrowlength=1.4,arrowinset=0.0]{->}(0.8689903,-2.5534792)(0.8689903,3.4465208)
\psdots[linecolor=black, dotsize=0.14](1.6689904,-1.7534791)
\psdots[linecolor=black, dotsize=0.14](2.4689903,-0.9534792)
\psdots[linecolor=black, dotsize=0.14](1.6689904,-0.15347916)
\psdots[linecolor=black, dotsize=0.14](2.4689903,0.64652085)
\psdots[linecolor=black, dotsize=0.14](3.2689903,-0.15347916)
\psdots[linecolor=black, dotsize=0.14](4.06899,-0.9534792)
\psdots[linecolor=black, dotsize=0.14](4.8689904,-0.15347916)
\psdots[linecolor=black, dotsize=0.14](5.66899,0.64652085)
\psdots[linecolor=black, dotsize=0.14](6.4689903,-0.15347916)
\psdots[linecolor=black, dotsize=0.14](0.8689903,-0.9534792)
\psdots[linecolor=black, dotsize=0.14](0.8689903,0.64652085)
\psdots[linecolor=black, dotsize=0.14](1.6689904,1.4465208)
\psdots[linecolor=black, dotsize=0.14](3.2689903,1.4465208)
\psdots[linecolor=black, dotsize=0.14](4.06899,0.64652085)
\psdots[linecolor=black, dotsize=0.14](4.8689904,1.4465208)
\psdots[linecolor=black, dotsize=0.14](2.4689903,-2.5534792)
\psdots[linecolor=black, dotsize=0.14](5.66899,-0.9534792)
\psdots[linecolor=black, dotsize=0.14](4.8689904,-1.7534791)
\psdots[linecolor=black, dotsize=0.14](3.2689903,-1.7534791)
\psdots[linecolor=black, dotsize=0.14](6.4689903,-1.7534791)
\psdots[linecolor=black, dotsize=0.14](0.06899033,1.4465208)
\psdots[linecolor=black, dotsize=0.14](0.06899033,-0.15347916)
\psdots[linecolor=black, dotsize=0.14](0.06899033,-1.7534791)
\psdots[linecolor=black, dotsize=0.14](0.8689903,-2.5534792)
\rput[bl](7.66899,-2.153479){$x$}
\rput[bl](0.46899033,3.046521){$y$}
\psdots[linecolor=black, dotsize=0.14](1.6689904,-0.9534792)
\psdots[linecolor=black, dotsize=0.14](0.8689903,-1.7534791)
\psdots[linecolor=black, dotsize=0.14](0.8689903,-0.15347916)
\psdots[linecolor=black, dotsize=0.14](0.8689903,1.4465208)
\psdots[linecolor=black, dotsize=0.14](0.06899033,0.64652085)
\psdots[linecolor=black, dotsize=0.14](1.6689904,0.64652085)
\psdots[linecolor=black, dotsize=0.14](2.4689903,-0.15347916)
\psdots[linecolor=black, dotsize=0.14](2.4689903,1.4465208)
\psdots[linecolor=black, dotsize=0.14](3.2689903,0.64652085)
\psdots[linecolor=black, dotsize=0.14](0.06899033,-0.9534792)
\psdots[linecolor=black, dotsize=0.14](4.06899,1.4465208)
\psdots[linecolor=black, dotsize=0.14](0.06899033,-2.5534792)
\psdots[linecolor=black, dotsize=0.14](5.66899,1.4465208)
\psdots[linecolor=black, dotsize=0.14](1.6689904,-2.5534792)
\psdots[linecolor=black, dotsize=0.14](6.4689903,0.64652085)
\psdots[linecolor=black, dotsize=0.14](4.8689904,0.64652085)
\psdots[linecolor=black, dotsize=0.14](4.06899,-0.15347916)
\psdots[linecolor=black, dotsize=0.14](3.2689903,-0.9534792)
\psdots[linecolor=black, dotsize=0.14](2.4689903,-1.7534791)
\psdots[linecolor=black, dotsize=0.14](4.06899,-1.7534791)
\psdots[linecolor=black, dotsize=0.14](4.8689904,-0.9534792)
\psdots[linecolor=black, dotsize=0.14](5.66899,-0.15347916)
\psdots[linecolor=black, dotsize=0.14](5.66899,-1.7534791)
\psdots[linecolor=black, dotsize=0.14](6.4689903,-0.9534792)
\rput[bl](0.46899033,-2.153479){$O$}
\psline[linecolor=black, linewidth=0.04](1.6689904,-0.15347916)(2.4689903,-0.9534792)(3.2689903,-1.7534791)(4.06899,-0.9534792)(3.2689903,-0.15347916)(4.06899,0.64652085)
\psdots[linecolor=black, dotsize=0.14](3.2689903,-2.5534792)
\psdots[linecolor=black, dotsize=0.14](4.06899,-2.5534792)
\psdots[linecolor=black, dotsize=0.14](4.8689904,-2.5534792)
\psdots[linecolor=black, dotsize=0.14](5.66899,-2.5534792)
\psdots[linecolor=black, dotsize=0.14](6.4689903,-2.5534792)
\psdots[linecolor=black, dotsize=0.14](6.4689903,1.4465208)
\psdots[linecolor=black, dotsize=0.14](0.06899033,2.2465208)
\psdots[linecolor=black, dotsize=0.14](0.8689903,2.2465208)
\psdots[linecolor=black, dotsize=0.14](1.6689904,2.2465208)
\psdots[linecolor=black, dotsize=0.14](2.4689903,2.2465208)
\psdots[linecolor=black, dotsize=0.14](3.2689903,2.2465208)
\psdots[linecolor=black, dotsize=0.14](4.06899,2.2465208)
\psdots[linecolor=black, dotsize=0.14](4.8689904,2.2465208)
\psdots[linecolor=black, dotsize=0.14](5.66899,2.2465208)
\psdots[linecolor=black, dotsize=0.14](6.4689903,2.2465208)
\end{pspicture}
}
\caption{The \textbf{X} -- path $(1,2),(2,1),(3,0),(4,1),(3,2),(4,3)$}
	\label{Fig 3}
\end{figure}
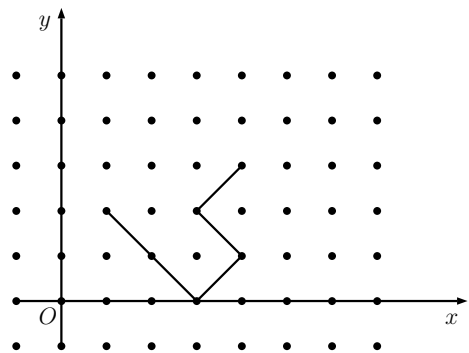

\begin{definition} An \textbf{X}-path is a lattice path  $v_{0} ,v_{1} ,\ldots ,v_{k} $ with steps in the set
\[S=\{ (1,1),(-1,1),(1,-1),(-1,-1)\} ,\] 
such that $v_{i} \ne v_{j} \; ,\; \forall i\ne j$.
\end{definition}
In particular, we consider an \textbf{X}-path to be the set of \textit{k} line segments connecting the points $v_{0} ,v_{1} ,\ldots ,v_{k} $. 

As can be understood, from a lattice point $(x,y)$ it is allowed to move in four directions and reach one of the four neighboring points $(x+1,y+1)$, $(x-1,y+1)$, $(x+1,y-1)$ and $(x-1,y-1)$. If $(x,y)$ is a starting point then any direction may be chosen, if not, we must take into account that $v_{i} \ne v_{j} \; ,\; \forall i\ne j$.

We notice that the sum of coordinates of any neighbor point has the same parity as $x+y$. We say that a lattice point $(x,y)$ is even (odd) if  $x+y$ is even (odd). An \textbf{X} -- path is even (odd) if it contains even (respectively, odd) points. For example, the \textbf{X} -- path represented in Fig. \ref{Fig 3} contains only odd points. 

Let $\mathcal{V}_{l,w} =\left\{A_{x,y} =(x,y)\in {\bf {\mathbb Z}}^{2} :0\le x\le l,\; 0\le y\le w\right\}$ be the set of all lattice points in the rectangle $[0,l]\times [0,w]$ and $V_{l,w}^{(1)} =\left\{A_{x,y} \in \mathcal{V}_{l,w} :x+y={\rm even}\right\}$, $V_{l,w}^{(2)} =\left\{A_{x,y} \in {\mathcal V}_{l,w} :x+y={\rm odd}\right\}$ be the subsets of even (respectively, odd) points in the rectangle $[0,l]\times [0,w]$. 

We denote by $\mathcal{E}_{l,w} =\left\{A_{x,y} A_{x',y'} :A_{x,y} ,A_{x',y'} \in \mathcal{V}_{l,w} ,\left|x-x'\right|=\left|y-y'\right|=1\right\}$ the set of all the line segments of length $\sqrt{2} $ connecting points of $\mathcal{V}_{l,w} $. Let $E_{l,w}^{(1)} =\left\{A_{x,y} A_{x',y'} \in {\mathcal E}_{l,w} :x+y={\rm even}\right\}$ be the subset of all even edges of $\mathcal{E}_{l,w}$ and let $E_{l,w}^{(2)} =\left\{A_{x,y} A_{x',y'} \in \mathcal{E}_{l,w} :x+y={\rm odd}\right\}$ be the subset of odd edges (the two disjoint subsets form a partition of $\mathcal{E}_{l,w} $).

A \textit{hammock network of the first kind }of dimensions $(l,w)$ is the probabilistic graph $H_{l,w}^{(1)} =\left(V_{l,w}^{(1)} ,E_{l,w}^{(1)} \right)$, while a \textit{hammock network of the second kind }is $H_{l,w}^{(2)} =\left(V_{l,w}^{(2)} ,E_{l,w}^{(2)} \right)$. We assume that each edge is closed with probability \textit{p} and open with probability $1-p$. The input (source) nodes are $S_{j} =A_{0,y} $ (with \textit{y} = even for the first kind and \textit{y} = odd for the second kind), and the output (terminus) nodes are $T_{k} =A_{l,z} $ (with $l+z$ = even, respectively, odd).

A subset of even (respectively, odd) edges $P\subset E_{l,w}^{(i)} $ is a \textit{pathset} in $H_{l,w}^{(i)} $ if it contains an \textbf{X -- }path connecting a source node $S_{j} $ with a target node $T_{k} $. Let $\mathcal{P}_{l,w}^{(i)} $ be the set of all pathsets in $H_{l,w}^{(i)} $. A subset $C\subset E_{l,w}^{(i)} $ is a \textit{cutset} in $H_{l,w}^{(i)} $ if $E_{l,w}^{(i)} -C$ contains no \textbf{X -- }path connecting a source node $S_{j} $ with a terminus node $T_{k} $. Let $\mathcal{C}_{l,w}^{(i)} $ be the set of all cutsets in $H_{l,w}^{(i)} $. By using these notations in formulas \eqref{1} and \eqref{2}, the reliability polynomials of hammock networks of the first and of second type, $h_{l,w}^{(1)} (p)$ and $h_{l,w}^{(2)} (p)$, can be written:
\begin{equation} \label{3} h_{l,w}^{(i)} (p)=\sum _{P\in \mathcal{P}_{l,w}^{(i)} }p^{|P|} (1-p)^{lw-|P|}  =1-\sum _{C\in \mathcal{C}_{l,w}^{(i)} }(1-p)^{|C|} p^{lw-|C|}  ,\; \, i=1,2 \end{equation}

\begin{figure}[h]\centering
\psscalebox{0.75} 
{
\begin{pspicture}(0,-2.6234791)(12.306958,2.6234791)
\psline[linecolor=black, linewidth=0.04, arrowsize=0.05291667cm 2.0,arrowlength=1.4,arrowinset=0.0]{->}(0.0,-1.6834792)(5.2,-1.6834792)
\psline[linecolor=black, linewidth=0.04, arrowsize=0.05291667cm 2.0,arrowlength=1.4,arrowinset=0.0]{->}(0.4,-2.0834792)(0.4,2.7165208)
\psdots[linecolor=black, dotsize=0.14](0.4,-1.6834792)
\psdots[linecolor=black, dotsize=0.14](1.2,-0.8834792)
\psdots[linecolor=black, dotsize=0.14](2.0,-1.6834792)
\psdots[linecolor=black, dotsize=0.14](2.8,-0.8834792)
\psdots[linecolor=black, dotsize=0.14](3.6,-1.6834792)
\psdots[linecolor=black, dotsize=0.14](0.4,-0.08347916)
\psdots[linecolor=black, dotsize=0.14](1.2,0.71652085)
\psdots[linecolor=black, dotsize=0.14](0.4,1.5165209)
\psdots[linecolor=black, dotsize=0.14](2.0,1.5165209)
\psdots[linecolor=black, dotsize=0.14](2.0,-0.08347916)
\psdots[linecolor=black, dotsize=0.14](3.6,-0.08347916)
\psdots[linecolor=black, dotsize=0.14](2.8,0.71652085)
\psdots[linecolor=black, dotsize=0.14](3.6,1.5165209)
\psline[linecolor=black, linewidth=0.04](0.4,-1.6834792)(1.2,-0.8834792)(2.0,-1.6834792)(2.8,-0.8834792)(3.6,-1.6834792)
\psline[linecolor=black, linewidth=0.04](0.4,-0.08347916)(1.2,-0.8834792)(2.0,-0.08347916)(2.8,-0.8834792)(3.6,-0.08347916)(2.8,0.71652085)(2.0,-0.08347916)(1.2,0.71652085)(0.4,-0.08347916)
\psline[linecolor=black, linewidth=0.04](0.4,1.5165209)(1.2,0.71652085)(2.0,1.5165209)(2.8,0.71652085)(3.6,1.5165209)
\rput[bl](0.0,-2.0834792){$O$}
\rput[bl](4.8,-2.0834792){$x$}
\rput[bl](0.0,2.316521){$y$}
\psline[linecolor=black, linewidth=0.04, arrowsize=0.05291667cm 2.0,arrowlength=1.4,arrowinset=0.0]{->}(7.2,-1.6834792)(12.4,-1.6834792)
\psline[linecolor=black, linewidth=0.04, arrowsize=0.05291667cm 2.0,arrowlength=1.4,arrowinset=0.0]{->}(7.6,-2.0834792)(7.6,2.7165208)
\psdots[linecolor=black, dotsize=0.14](7.6,-0.8834792)
\psdots[linecolor=black, dotsize=0.14](8.4,-1.6834792)
\psdots[linecolor=black, dotsize=0.14](9.2,-0.8834792)
\psdots[linecolor=black, dotsize=0.14](10.0,-1.6834792)
\psdots[linecolor=black, dotsize=0.14](8.4,1.5165209)
\psdots[linecolor=black, dotsize=0.14](7.6,0.71652085)
\psdots[linecolor=black, dotsize=0.14](10.0,1.5165209)
\psdots[linecolor=black, dotsize=0.14](8.4,-0.08347916)
\psdots[linecolor=black, dotsize=0.14](10.8,-0.8834792)
\psdots[linecolor=black, dotsize=0.14](10.0,-0.08347916)
\psdots[linecolor=black, dotsize=0.14](10.8,0.71652085)
\rput[bl](7.2,-2.0834792){$O$}
\rput[bl](12.0,-2.0834792){$x$}
\rput[bl](7.2,2.316521){$y$}
\psdots[linecolor=black, dotsize=0.14](9.2,0.71652085)
\psline[linecolor=black, linewidth=0.04](7.6,-0.8834792)(8.4,-1.6834792)(9.2,-0.8834792)(10.0,-1.6834792)(10.8,-0.8834792)(10.0,-0.08347916)(9.2,-0.8834792)(8.4,-0.08347916)(7.6,-0.8834792)
\psline[linecolor=black, linewidth=0.04](7.6,0.71652085)(8.4,-0.08347916)(9.2,0.71652085)(10.0,-0.08347916)(10.8,0.71652085)(10.0,1.5165209)(9.2,0.71652085)(8.4,1.5165209)(7.6,0.71652085)
\rput[bl](2.0,-2.0634792){2}
\rput[bl](3.6,-2.0834792){$4$}
\rput[bl](0.0,-0.20347916){$2$}
\rput[bl](0.0,1.3765209){$4$}
\rput[bl](8.4,-2.0834792){$1$}
\rput[bl](10.0,-2.0834792){$3$}
\rput[bl](7.2,-0.8834792){$1$}
\rput[bl](7.2,0.71652085){$3$}
\rput[bl](2.4,-2.6234791){(a)}
\rput[bl](9.2,-2.6034791){(b)}
\end{pspicture}
}
\caption{Hammock networks of the first kind (a) and of the second kind (b)}
	\label{Fig 4}
\end{figure}
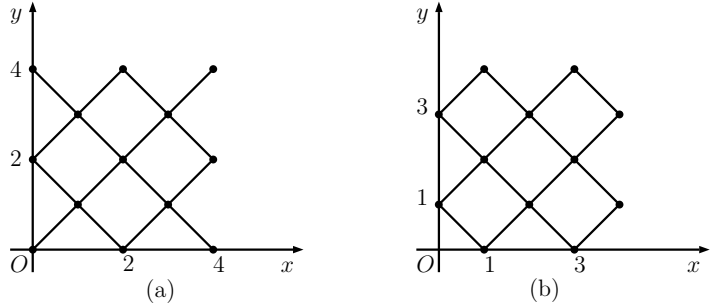

\begin{remark}\label{r1} {\rm If \textit{l }= odd or \textit{w} = odd, then the hammock networks $H_{l,w}^{(1)} $ and $H_{l,w}^{(2)} $ are isomorphic and the reliability polynomials are identical: $h_{l,w}^{(1)} =h_{l,w}^{(2)} $. If \textit{l} and \textit{w} are both even numbers, then we have two different networks of dimensions $(l,w)$: $h_{l,w}^{(1)} \ne h_{l,w}^{(2)} $.  For example, in Fig. \ref{Fig 4} the hammock networks of the first type and second type of dimensions $w=l=4$ are represented (corresponding to the brick-wall networks from Fig. \ref{Fig 1}). }\end{remark}

\section{Dual network}

For every edge $e\in {\rm {\mathcal E}}_{l,w} $, $e=A_{x,y} A_{x+1,y\pm 1} $, we denote by $\bar{e}=A_{x+1,y} A_{x,y\pm 1} $ its complementary edge (the edge that \textit{cuts e}). It can be seen that the complementary edge of an even edge is odd and the complementary edge of an odd edge is even. Thus, if $e\in E_{l.w}^{(i)} $, then $\bar{e}\in \overline{E_{l,w}^{(i)} }={\rm {\mathcal E}}_{l,w} -E_{l,w}^{(i)} =E_{l,w}^{(2/i)} $. By using the notation $\overline{V_{l,w}^{(i)} }={\rm {\mathcal V}}_{l,w} -V_{l,w}^{(i)} =V_{l,w}^{(2/i)} $, the dual network of $H_{l,w}^{(i)} =\left(V_{l,w}^{(i)} ,E_{l,w}^{(i)} \right)$ is $\overline{H_{l,w}^{(i)} }=\left(\overline{V_{l,w}^{(i)} },\overline{E_{l,w}^{(i)} }\right)$ with the source nodes $S'_{j} =A_{x,0} \in \overline{V_{l,w}^{(i)} }$ and the terminus nodes $T'_{k} =A_{z,w} \in \overline{V_{l,w}^{(i)} }$ (see Fig. \ref{Fig 5}). The probability of an edge $\bar{e}\in \overline{E_{l,w}^{(i)} }$ being closed is the probability of the edge $e\in E_{l,w}^{(i)} $ being open (``cut''): $q=1-p$. 

\begin{remark}\label{r2} {\rm Since for every edge $e\in {\rm {\mathcal E}}_{l,w} $, $\bar{\bar{e}}=e$ it follows that $\overline{\overline{H_{l,w}^{(i)} }}=H_{l,w}^{(i)} $.}\end{remark}

\begin{remark}\label{r3} {\rm The networks $\overline{H_{l,w}^{(i)} }$ and $H_{w,l}^{(2/i)} $  are isomorphic (since they are symmetric with respect to the first bisectrix) and the reliability polynomial of the dual network $\overline{H_{l,w}^{(i)} }$ can be written
\begin{equation}\label{4}
 \overline{h_{l,w}^{(i)} }(p)=h_{w,l}^{(2/i)} (1-p). \end{equation} }\end{remark}

 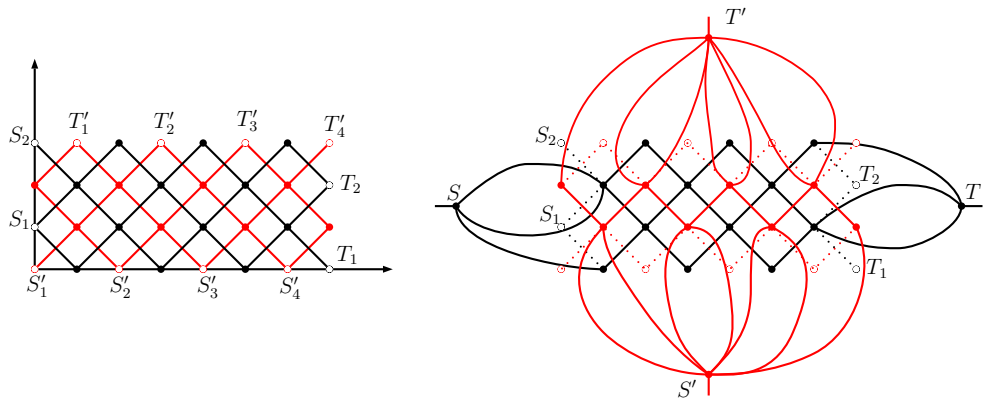
\begin{figure}[h]\centering
 \psscalebox{0.7} 
{
\begin{pspicture}(0,-3.67)(18.5,3.67)
\psline[linecolor=black, linewidth=0.04, linestyle=dotted, dotsep=0.10583334cm](11.3,0.35)(10.5,1.15)
\psline[linecolor=black, linewidth=0.04, linestyle=dotted, dotsep=0.10583334cm](11.3,0.35)(10.5,-0.45)(11.3,-1.25)
\psline[linecolor=black, linewidth=0.04, linestyle=dotted, dotsep=0.10583334cm](15.3,1.15)(16.1,0.35)(15.3,-0.45)(16.1,-1.25)
\psbezier[linecolor=black, linewidth=0.04](8.5,-0.05)(8.5,-0.05)(9.3,0.75)(9.86,0.75)(10.42,0.75)(10.82,0.75)(11.3,0.35)
\psbezier[linecolor=black, linewidth=0.04](8.5,-0.05)(8.966666,-0.85)(11.3,-0.85)(11.3,0.35)
\psbezier[linecolor=black, linewidth=0.04](8.5,-0.05)(8.5,-0.85)(10.5,-1.25)(11.3,-1.25)
\psline[linecolor=black, linewidth=0.04](11.3,0.35)(12.1,1.15)(12.9,0.35)(13.7,1.15)(14.5,0.35)(15.3,1.15)
\psline[linecolor=black, linewidth=0.04](11.3,0.35)(12.1,-0.45)(12.9,0.35)(13.7,-0.45)(14.5,0.35)(15.3,-0.45)
\psline[linecolor=black, linewidth=0.04](11.3,-1.25)(12.1,-0.45)(12.9,-1.25)(13.7,-0.45)(14.5,-1.25)(15.3,-0.45)
\psdots[linecolor=black, dotsize=0.14](8.5,-0.05)
\psdots[linecolor=black, dotsize=0.14](11.3,0.35)
\psdots[linecolor=black, dotsize=0.14](12.1,1.15)
\psdots[linecolor=black, dotsize=0.14](12.9,0.35)
\psdots[linecolor=black, dotsize=0.14](12.1,-0.45)
\psdots[linecolor=black, dotsize=0.14](11.3,-1.25)
\psdots[linecolor=black, dotsize=0.14](12.9,-1.25)
\psdots[linecolor=black, dotsize=0.14](13.7,1.15)
\psdots[linecolor=black, dotsize=0.14](13.7,-0.45)
\psdots[linecolor=black, dotsize=0.14](14.5,0.35)
\psdots[linecolor=black, dotsize=0.14](14.5,-1.25)
\psdots[linecolor=black, dotsize=0.14](15.3,-0.45)
\psdots[linecolor=black, dotsize=0.14](15.3,1.15)
\psdots[linecolor=black, fillstyle=solid, dotstyle=o, dotsize=0.14, fillcolor=white](10.5,1.15)
\psdots[linecolor=black, fillstyle=solid, dotstyle=o, dotsize=0.14, fillcolor=white](10.5,-0.45)
\psdots[linecolor=black, fillstyle=solid, dotstyle=o, dotsize=0.14, fillcolor=white](16.1,0.35)
\psdots[linecolor=black, fillstyle=solid, dotstyle=o, dotsize=0.14, fillcolor=white](16.1,-1.25)
\psbezier[linecolor=black, linewidth=0.04](15.3,1.15)(16.1,1.15)(16.5,1.15)(17.3,0.75)(18.1,0.35)(18.1,-0.05)(18.1,-0.05)
\psbezier[linecolor=black, linewidth=0.04](15.3,-0.45)(16.1,-0.85)(16.5,-0.85)(16.9,-0.85)(17.3,-0.85)(18.1,-0.45)(18.1,-0.05)
\psbezier[linecolor=black, linewidth=0.04](15.3,-0.45)(15.7,-0.05)(16.5,0.35)(16.9,0.35)(17.3,0.35)(17.7,0.35)(18.1,-0.05)
\psdots[linecolor=black, dotsize=0.14](18.1,-0.05)
\rput[bl](8.32,0.07){$S$}
\rput[bl](18.18,0.09){$T$}
\rput[bl](10.0,1.11){$S_2$}
\rput[bl](10.08,-0.39){$S_1$}
\rput[bl](16.16,0.37){$T_2$}
\rput[bl](16.3,-1.41){$T_1$}
\psdots[linecolor=red, dotsize=0.14](12.1,0.35)
\psdots[linecolor=red, dotsize=0.14](13.7,0.35)
\psdots[linecolor=red, dotsize=0.14](15.3,0.35)
\psdots[linecolor=red, dotsize=0.14](10.5,0.35)
\psdots[linecolor=red, dotsize=0.14](11.3,-0.45)
\psdots[linecolor=red, dotsize=0.14](12.9,-0.45)
\psdots[linecolor=red, dotsize=0.14](14.5,-0.45)
\psdots[linecolor=red, dotsize=0.14](16.1,-0.45)
\psline[linecolor=black, linewidth=0.04](8.5,-0.05)(8.1,-0.05)
\psdots[linecolor=red, fillstyle=solid, dotstyle=o, dotsize=0.14, fillcolor=white](10.5,-1.25)
\psdots[linecolor=red, fillstyle=solid, dotstyle=o, dotsize=0.14, fillcolor=white](12.1,-1.25)
\psdots[linecolor=red, fillstyle=solid, dotstyle=o, dotsize=0.14, fillcolor=white](13.7,-1.25)
\psdots[linecolor=red, fillstyle=solid, dotstyle=o, dotsize=0.14, fillcolor=white](15.3,-1.25)
\psdots[linecolor=red, fillstyle=solid, dotstyle=o, dotsize=0.14, fillcolor=white](14.5,1.15)
\psdots[linecolor=red, fillstyle=solid, dotstyle=o, dotsize=0.14, fillcolor=white](16.1,1.15)
\psdots[linecolor=red, fillstyle=solid, dotstyle=o, dotsize=0.14, fillcolor=white](12.9,1.15)
\psdots[linecolor=red, fillstyle=solid, dotstyle=o, dotsize=0.14, fillcolor=white](11.3,1.15)
\psline[linecolor=red, linewidth=0.04](10.5,0.35)(11.3,-0.45)(12.1,0.35)(12.9,-0.45)(13.7,0.35)(14.5,-0.45)(15.3,0.35)(16.1,-0.45)
\psline[linecolor=red, linewidth=0.04, linestyle=dotted, dotsep=0.10583334cm](10.5,0.35)(11.3,1.15)(12.1,0.35)(12.9,1.15)(13.7,0.35)(14.5,1.15)(15.3,0.35)(16.1,1.15)
\psline[linecolor=red, linewidth=0.04, linestyle=dotted, dotsep=0.10583334cm](10.5,-1.25)(11.3,-0.45)(12.1,-1.25)(12.9,-0.45)(13.7,-1.25)(14.5,-0.45)(15.3,-1.25)(16.1,-0.45)
\psbezier[linecolor=red, linewidth=0.036](10.5,0.35)(10.5,0.75)(10.592894,1.6428932)(11.3,2.35)(12.007107,3.0571067)(12.5,3.15)(13.3,3.15)
\psbezier[linecolor=red, linewidth=0.036](12.1,0.35)(11.7,0.35)(11.3,0.75)(11.7,1.55)(12.1,2.35)(12.9,2.75)(13.3,3.15)
\psbezier[linecolor=red, linewidth=0.036](12.1,0.35)(12.1,0.35)(12.506081,0.63085496)(12.9,1.55)(13.29392,2.469145)(13.3,3.15)(13.3,3.15)
\psbezier[linecolor=red, linewidth=0.036](13.7,0.35)(13.3,0.75)(13.4414215,0.5600505)(13.3,1.55)(13.158579,2.5399494)(13.3,2.75)(13.3,3.15)
\psbezier[linecolor=red, linewidth=0.036](13.7,0.35)(14.1,0.35)(14.1,0.75)(14.1,1.15)(14.1,1.55)(13.7,2.35)(13.3,3.15)
\psbezier[linecolor=red, linewidth=0.036](15.3,0.35)(14.9,0.35)(14.5,1.15)(14.5,1.15)(14.5,1.15)(13.7,2.75)(13.3,3.15)
\psbezier[linecolor=red, linewidth=0.036](15.3,0.35)(15.7,1.15)(15.7,1.15)(15.7,1.55)(15.7,1.95)(15.3,2.35)(15.3,2.35)(15.3,2.35)(14.5,3.15)(13.3,3.15)
\psdots[linecolor=red, dotsize=0.14](13.3,3.15)
\psline[linecolor=black, linewidth=0.04](18.1,-0.05)(18.5,-0.05)
\rput[bl](13.62,3.41){$T'$}
\psbezier[linecolor=red, linewidth=0.04](11.3,-0.45)(10.5,-1.25)(10.9,-2.05)(11.3,-2.45)(11.7,-2.85)(12.5,-3.25)(13.3,-3.25)
\psbezier[linecolor=red, linewidth=0.04](11.3,-0.45)(11.3,-0.85)(11.518762,-1.2362665)(12.1,-2.05)(12.681238,-2.8637335)(13.3,-3.25)(13.3,-3.25)
\psbezier[linecolor=red, linewidth=0.04](12.9,-0.45)(12.9,-0.45)(12.358579,-0.6600505)(12.5,-1.65)(12.641421,-2.6399496)(12.9,-2.85)(13.3,-3.25)
\psbezier[linecolor=red, linewidth=0.04](12.9,-0.45)(13.3,-0.45)(13.5585785,-0.6600505)(13.7,-1.65)(13.841421,-2.6399496)(13.7,-2.85)(13.3,-3.25)
\psbezier[linecolor=red, linewidth=0.04](14.5,-0.45)(14.1,-0.45)(14.1,-0.85)(14.1,-1.65)(14.1,-2.45)(13.7,-2.85)(13.3,-3.25)
\psbezier[linecolor=red, linewidth=0.04](14.5,-0.45)(15.3,-0.85)(15.29392,-1.530855)(14.9,-2.45)(14.506081,-3.369145)(13.7,-3.25)(13.3,-3.25)
\psbezier[linecolor=red, linewidth=0.04](16.1,-0.45)(16.5,-1.25)(16.1,-2.45)(15.3,-2.85)(14.5,-3.25)(13.7,-3.25)(13.3,-3.25)
\psdots[linecolor=red, dotsize=0.14](13.3,-3.25)
\rput[bl](12.7,-3.67){$S'$}
\psline[linecolor=red, linewidth=0.04](13.3,3.15)(13.3,3.55)
\psline[linecolor=red, linewidth=0.04](13.3,-3.25)(13.3,-3.65)
\psline[linecolor=black, linewidth=0.04, arrowsize=0.05291667cm 2.0,arrowlength=1.4,arrowinset=0.0]{->}(0.5,-1.25)(7.3,-1.25)
\psline[linecolor=black, linewidth=0.04, arrowsize=0.05291667cm 2.0,arrowlength=1.4,arrowinset=0.0]{->}(0.5,-1.25)(0.5,2.75)
\psdots[linecolor=red, dotsize=0.14](0.5,-1.25)
\psdots[linecolor=red, dotsize=0.14](1.3,-0.45)
\psdots[linecolor=red, dotsize=0.14](2.1,-1.25)
\psdots[linecolor=red, dotsize=0.14](2.9,-0.45)
\psdots[linecolor=red, dotsize=0.14](3.7,-1.25)
\psdots[linecolor=red, dotsize=0.14](4.5,-0.45)
\psdots[linecolor=red, dotsize=0.14](5.3,-1.25)
\psdots[linecolor=red, dotsize=0.14](5.3,0.35)
\psdots[linecolor=red, dotsize=0.14](4.5,1.15)
\psdots[linecolor=red, dotsize=0.14](3.7,0.35)
\psdots[linecolor=red, dotsize=0.14](2.1,0.35)
\psdots[linecolor=red, dotsize=0.14](0.5,0.35)
\psdots[linecolor=black, dotsize=0.14](0.5,-0.45)
\psdots[linecolor=black, dotsize=0.14](1.3,-1.25)
\psdots[linecolor=black, dotsize=0.14](2.1,-0.45)
\psdots[linecolor=black, dotsize=0.14](2.9,-1.25)
\psdots[linecolor=black, dotsize=0.14](3.7,-0.45)
\psdots[linecolor=black, dotsize=0.14](4.5,-1.25)
\psdots[linecolor=black, dotsize=0.14](5.3,-0.45)
\psdots[linecolor=black, dotsize=0.14](4.5,0.35)
\psdots[linecolor=black, dotsize=0.14](5.3,1.15)
\psdots[linecolor=black, dotsize=0.14](3.7,1.15)
\psdots[linecolor=black, dotsize=0.14](2.9,0.35)
\psdots[linecolor=black, dotsize=0.14](2.1,1.15)
\psdots[linecolor=black, dotsize=0.14](1.3,0.35)
\psdots[linecolor=black, dotsize=0.14](0.5,1.15)
\rput[bl](0.0,1.11){$S_2$}
\rput[bl](0.0,-0.49){$S_1$}
\rput[bl](6.3,0.19){$T_2$}
\rput[bl](6.24,-1.15){$T_1$}
\rput[bl](0.34,-1.77){$S'_1$}
\rput[bl](1.92,-1.79){$S'_2$}
\rput[bl](3.58,-1.79){$S'_3$}
\rput[bl](5.14,-1.81){$S'_4$}
\rput[bl](1.14,1.33){$T'_1$}
\rput[bl](2.76,1.33){$T'_2$}
\rput[bl](4.34,1.37){$T'_3$}
\rput[bl](6.0,1.27){$T'_4$}
\psdots[linecolor=red, dotsize=0.14](6.1,-0.45)
\psline[linecolor=red, linewidth=0.04](0.5,-1.25)(1.3,-0.45)(2.1,-1.25)(2.9,-0.45)(3.7,-1.25)(4.5,-0.45)(5.3,-1.25)(6.1,-0.45)(5.3,0.35)(4.5,-0.45)(3.7,0.35)(2.9,-0.45)(2.1,0.35)(1.3,-0.45)(0.5,0.35)(1.3,1.15)(2.1,0.35)(2.9,1.15)(3.7,0.35)(4.5,1.15)(5.3,0.35)(6.1,1.15)
\psdots[linecolor=red, dotstyle=o, dotsize=0.14, fillcolor=white](1.3,1.15)
\psdots[linecolor=red, dotstyle=o, dotsize=0.14, fillcolor=white](1.3,1.15)
\psdots[linecolor=red, dotsize=0.14](2.9,1.15)
\psdots[linecolor=red, dotstyle=o, dotsize=0.14, fillcolor=white](2.9,1.15)
\psdots[linecolor=red, dotstyle=o, dotsize=0.14, fillcolor=white](4.5,1.15)
\psdots[linecolor=red, dotstyle=o, dotsize=0.14, fillcolor=white](6.1,1.15)
\psdots[linecolor=red, dotstyle=o, dotsize=0.14, fillcolor=white](0.5,-1.25)
\psdots[linecolor=red, dotstyle=o, dotsize=0.14, fillcolor=white](2.1,-1.25)
\psdots[linecolor=red, dotstyle=o, dotsize=0.14, fillcolor=white](3.7,-1.25)
\psdots[linecolor=red, dotstyle=o, dotsize=0.14, fillcolor=white](5.3,-1.25)
\psline[linecolor=black, linewidth=0.04](6.1,-1.25)(5.3,-0.45)(4.5,-1.25)(3.7,-0.45)(2.9,-1.25)(2.1,-0.45)(1.3,-1.25)(0.5,-0.45)(1.3,0.35)(2.1,-0.45)(2.9,0.35)(3.7,-0.45)(4.5,0.35)(5.3,-0.45)(6.1,0.35)(5.3,1.15)(4.5,0.35)(3.7,1.15)(2.9,0.35)(2.1,1.15)(1.3,0.35)(0.5,1.15)
\psdots[linecolor=black, dotstyle=o, dotsize=0.14, fillcolor=white](0.5,1.15)
\psdots[linecolor=black, dotstyle=o, dotsize=0.14, fillcolor=white](0.5,-0.45)
\psdots[linecolor=black, dotstyle=o, dotsize=0.14, fillcolor=white](6.1,-1.25)
\psdots[linecolor=black, dotstyle=o, dotsize=0.14, fillcolor=white](6.1,0.35)
\end{pspicture}
}
\caption{Dual networks}
	\label{Fig 5}
\end{figure}

Let $G_{l,w}^{(i)} $ be the graph obtained from $H_{l,w}^{(i)} $ by replacing the ``fictive'' nodes $S_{1} ,S_{2} ,\ldots, S_k $ and  $T_{1} ,T_{2} ,\ldots, T_h $ with the nodes \textit{S } and \textit{T}, respectively, and let $\overline{G_{l,w}^{(i)} }$ be the graph obtained from $\overline{H_{l,w}^{(i)} }$ by the same operation (the terminal nodes, in this case, are $S'$ and $T'$). We notice that, if we consider the terminal nodes \textit{S }and \textit{T} as being placed to $\pm \infty $, then $\overline{G_{l,w}^{(i)} }$ as defined above corresponds to the definition of the dual graph of $G_{l,w}^{(i)} $. In Fig. \ref{Fig 5} the hammock network of dimensions $w=3$, $l=7$ (from Fig. \ref{Fig 2}), and its dual network (red), are presented.

\section{The reliability polynomial of a hammock network}

The main result of this paper is represented by Theorem \ref{t1} whose corollaries make the connection between the reliability polynomials of a hammock network and its dual. The proof of this theorem relies on the \textbf{Jordan Curve Theorem} \cite{Ha} which states that \textit{every simple closed plane curve divides the plane into an "interior" region bounded by the curve and an "exterior" region, so that every continuous path connecting a point from one region to a point from the other intersects that curve somewhere.}

\begin{theorem}\label{t1} Let $\Sigma =\left\{e_{1} ,e_{2} ,\ldots ,e_{n} \right\}\subset E_{l,w}^{(i)} $ be a subset of edges of the network $H_{l,w}^{(i)} $ and let $\overline{\Sigma }=\left\{\bar{e}_{1} ,\bar{e}_{2} ,\ldots ,\bar{e}_{n} \right\}\subset \overline{E_{l,w}^{(i)} }$ be the set of complementary edges ($i=1,2$). Then the following statements hold:

i) If $\Sigma $ is a \textit{mincut} in $H_{l,w}^{(i)}$, then $\overline{\Sigma }$ is a \textit{minpath} in $\overline{H_{l,w}^{(i)} }$;

ii) If $\Sigma $ is a \textit{minpath} in $H_{l,w}^{(i)}$, then $\overline{\Sigma }$ is a \textit{mincut} in $\overline{H_{l,w}^{(i)} }$.

\end{theorem}

\begin{proof} 

i) Since $\Sigma $ is a mincut, for every $e_{i} \in \Sigma $ there exists an \textbf{X }-- path which contains $e_{i} $ and connects a source node (denoted by $S_{i} $) to a terminus node (denoted by $T_{i} $): $\xi _{i} =\sigma _{i} \cup e_{i} \cup \tau _{i} $, where  $\sigma _{i} $ is an \textbf{X }-- path from $S_{i} $ to $e_{i} $ and $\tau _{i} $ is an \textbf{X }-- path from $e_{i} $ to $T_{i} $  and $\sigma _{i} ,\tau _{i} \subset E_{l,w}^{(i)} -\Sigma $. Let $E_{i} ,F_{i} $ be the end vertices of $e_{i} $, where $E_{i} $ is reachable from the source node $S_{i} $ and $F_{i} $ from the target node $T_{i} $. We can see that $\sigma _{i} \cap \tau _{j} =\emptyset $ for $i\ne j$ (otherwise $\Sigma $ would not be a cutset). As a consequence, $E_{i} \ne F_{j} $ for $i\ne j$. Obviously, in order to be a cutset, $\Sigma $ must contain an edge with a vertex on \textit{Ox} axis. It can be proved that $\Sigma $ cannot contain two such edges. Suppose $e_{i} ,e_{j} \in \Sigma $ are two edges with a vertex on \textit{Ox} and $e_{i} $ is closer to \textit{O} then $e_{j} $. If $F_i\in Ox$, we consider the simple closed curve $\gamma =\tau _{i} \cup T_{i} A_{l+1,-1} \cup A_{l+1,-1}F_i $  (see Fig. \ref{Fig 6}), otherwise we take $\gamma =\tau _{i} \cup T_{i} A_{l+1,-1} \cup A_{l+1,-1} F'_{i}\cup F'_{i}F_i $, where $F'_i$ is the projection of $F_i$ on $Ox$ axis. We notice that $E_{j} $ belongs to the interior domain region bounded by $\gamma $ (otherwise, if $E_{j}$ was on $\tau _{i} $, we would have $\sigma _{j} \cap \tau _{i} \ne \emptyset $). Since $E_{j} $ is an interior point and $S_{j} $ is an exterior point of $\gamma $, it follows (by Jordan curve theorem) that the continuous path $\sigma _{j} $ connecting the two points intersects $\gamma $ somewhere, so $\sigma _{j} \cap \tau _{i} \ne \emptyset $, which is impossible (see Fig. \ref{Fig 6}). It follows that $\Sigma $ contains exactly one edge with a vertex on the \textit{Ox} axis. Let $e_{I} $ be this ``initial'' edge. Similarly, $\Sigma $ contains exactly one edge with a vertex on the straight line $y=w$, and let $e_{F} $ be this final edge.

 \begin{figure}[h]\centering
\psscalebox{0.8} 
{
\begin{pspicture}(0,-3.903479)(12.94,3.903479)
\psline[linecolor=black, linewidth=0.04, arrowsize=0.05291667cm 2.0,arrowlength=1.4,arrowinset=0.0]{->}(0.08,-2.8034792)(12.88,-2.8034792)
\psline[linecolor=black, linewidth=0.04, arrowsize=0.05291667cm 2.0,arrowlength=1.4,arrowinset=0.0]{->}(0.48,-3.6034791)(0.48,3.9965208)
\psdots[linecolor=black, dotsize=0.1](0.48,-2.8034792)
\psdots[linecolor=black, dotsize=0.1](2.08,-2.8034792)
\psdots[linecolor=black, dotsize=0.1](3.68,-2.8034792)
\psdots[linecolor=black, dotsize=0.1](5.28,-2.8034792)
\psdots[linecolor=black, dotsize=0.1](6.88,-2.8034792)
\psdots[linecolor=black, dotsize=0.1](8.48,-2.8034792)
\psdots[linecolor=black, dotsize=0.1](10.08,-2.8034792)
\psdots[linecolor=black, dotsize=0.1](1.28,-2.0034792)
\psdots[linecolor=black, dotsize=0.1](2.88,-2.0034792)
\psdots[linecolor=black, dotsize=0.1](4.48,-2.0034792)
\psdots[linecolor=black, dotsize=0.1](6.08,-2.0034792)
\psdots[linecolor=black, dotsize=0.1](7.68,-2.0034792)
\psdots[linecolor=black, dotsize=0.1](9.28,-2.0034792)
\psdots[linecolor=black, dotsize=0.1](8.48,-1.2034792)
\psdots[linecolor=black, dotsize=0.1](6.88,-1.2034792)
\psdots[linecolor=black, dotsize=0.1](5.28,-1.2034792)
\psdots[linecolor=black, dotsize=0.1](3.68,-1.2034792)
\psdots[linecolor=black, dotsize=0.1](2.08,-1.2034792)
\psdots[linecolor=black, dotsize=0.1](0.48,-1.2034792)
\psdots[linecolor=black, dotsize=0.1](7.68,-0.40347916)
\psdots[linecolor=black, dotsize=0.1](6.08,-0.40347916)
\psdots[linecolor=black, dotsize=0.1](4.48,-0.40347916)
\psdots[linecolor=black, dotsize=0.1](2.88,-0.40347916)
\psdots[linecolor=black, dotsize=0.1](1.28,-0.40347916)
\psdots[linecolor=black, dotsize=0.1](9.28,-0.40347916)
\psdots[linecolor=black, dotsize=0.1](9.28,1.1965208)
\psdots[linecolor=black, dotsize=0.1](9.28,2.796521)
\psdots[linecolor=black, dotsize=0.1](8.48,0.39652085)
\psdots[linecolor=black, dotsize=0.1](6.88,0.39652085)
\psdots[linecolor=black, dotsize=0.1](7.68,1.1965208)
\psdots[linecolor=black, dotsize=0.1](8.48,1.9965209)
\psdots[linecolor=black, dotsize=0.1](0.48,0.39652085)
\psdots[linecolor=black, dotsize=0.1](0.48,1.9965209)
\psdots[linecolor=black, dotsize=0.1](1.28,2.796521)
\psdots[linecolor=black, dotsize=0.1](1.28,1.1965208)
\psdots[linecolor=black, dotsize=0.1](2.08,1.9965209)
\psdots[linecolor=black, dotsize=0.1](2.88,2.796521)
\psdots[linecolor=black, dotsize=0.1](2.08,0.39652085)
\psdots[linecolor=black, dotsize=0.1](2.88,1.1965208)
\psdots[linecolor=black, dotsize=0.1](3.68,1.9965209)
\psdots[linecolor=black, dotsize=0.1](4.48,2.796521)
\psdots[linecolor=black, dotsize=0.1](3.68,0.39652085)
\psdots[linecolor=black, dotsize=0.1](4.48,1.1965208)
\psdots[linecolor=black, dotsize=0.1](5.28,1.9965209)
\psdots[linecolor=black, dotsize=0.1](6.08,2.796521)
\psdots[linecolor=black, dotsize=0.1](5.28,0.39652085)
\psdots[linecolor=black, dotsize=0.1](6.08,1.1965208)
\psdots[linecolor=black, dotsize=0.1](6.88,1.9965209)
\psdots[linecolor=black, dotsize=0.1](7.68,2.796521)
\psdots[linecolor=black, dotsize=0.1](10.08,-1.2034792)
\psdots[linecolor=black, dotsize=0.1](10.08,0.39652085)
\psdots[linecolor=black, dotsize=0.1](10.08,1.9965209)
\psdots[linecolor=black, dotsize=0.1](10.88,-2.0034792)
\psdots[linecolor=black, dotsize=0.1](10.88,-0.40347916)
\psdots[linecolor=black, dotsize=0.1](10.88,1.1965208)
\psdots[linecolor=black, dotsize=0.1](10.88,2.796521)
\psline[linecolor=red, linewidth=0.07](4.48,-2.0034792)(5.28,-2.8034792)
\psline[linecolor=red, linewidth=0.07](6.88,-2.8034792)(7.68,-2.0034792)
\rput[bl](0.08,3.596521){$y$}
\rput[bl](0.12,-3.183479){$O$}
\rput[bl](12.48,-3.203479){$x$}
\rput[bl](0.0,-1.3034792){$S_i$}
\rput[bl](0.0,1.7365209){$S_j$}
\rput[bl](4.94,-0.86347914){$\sigma_j$}
\rput[bl](11.44,-2.463479){$\gamma$}
\rput[bl](7.0,-0.74347913){$\tau_i$}
\rput[bl](1.7,-2.3434792){$\sigma_i$}
\rput[bl](8.48,-2.403479){$\tau_j$}
\rput[bl](11.82,-3.903479){$A_{l+1,-1}$}
\rput[bl](4.3,-1.8034792){$E_i$}
\rput[bl](5.04,-3.2434793){$F_i$}
\rput[bl](7.7,-2.0234792){$F_j$}
\rput[bl](6.7,-2.5634792){$E_j$}
\rput[bl](11.1,-0.54347914){$T_i$}
\psline[linecolor=red, linewidth=0.03](4.48,-2.0034792)(3.68,-1.2034792)
\psline[linecolor=red, linewidth=0.03](4.48,-2.0034792)(5.28,-1.2034792)
\psline[linecolor=red, linewidth=0.03](3.68,-1.2034792)(2.88,-2.0034792)
\psline[linecolor=red, linewidth=0.03](2.08,-1.2034792)(2.88,-0.40347916)
\psline[linecolor=red, linewidth=0.03](1.28,-0.40347916)(2.08,0.39652085)
\psline[linecolor=red, linewidth=0.03](1.28,1.1965208)(2.08,0.39652085)
\psline[linecolor=red, linewidth=0.03](2.08,1.9965209)(2.88,1.1965208)(3.68,1.9965209)
\psline[linecolor=red, linewidth=0.03](4.48,-0.40347916)(5.28,0.39652085)(6.08,-0.40347916)(6.88,0.39652085)(7.68,-0.40347916)
\psline[linecolor=red, linewidth=0.03](7.68,1.1965208)(8.48,0.39652085)(10.08,1.9965209)
\psline[linecolor=red, linewidth=0.03](9.28,-0.40347916)(10.88,1.1965208)
\psline[linecolor=red, linewidth=0.03](10.08,0.39652085)(10.88,-0.40347916)
\psline[linecolor=red, linewidth=0.03](8.48,1.9965209)(9.28,2.796521)
\psline[linecolor=black, linewidth=0.04, linestyle=dashed, dash=0.17638889cm 0.10583334cm](4.48,-2.0034792)(3.68,-2.8034792)(2.88,-2.0034792)(2.08,-2.8034792)(0.48,-1.2034792)
\psline[linecolor=black, linewidth=0.04, linestyle=dashed, dash=0.17638889cm 0.10583334cm](5.28,-2.8034792)(7.68,-0.40347916)(8.48,0.39652085)(10.08,-1.2034792)(10.88,-0.40347916)
\psline[linecolor=black, linewidth=0.04, linestyle=dashed, dash=0.17638889cm 0.10583334cm](7.68,-2.0034792)(8.48,-2.8034792)(9.28,-2.0034792)(10.08,-2.8034792)(10.88,-2.0034792)
\psline[linecolor=black, linewidth=0.04, linestyle=dashed, dash=0.17638889cm 0.10583334cm](6.88,-2.8034792)(3.68,0.39652085)(5.28,1.9965209)(4.48,2.796521)(3.68,1.9965209)(2.88,2.796521)(2.08,1.9965209)(1.28,2.796521)(0.48,1.9965209)
\psline[linecolor=red, linewidth=0.03](3.68,1.9965209)(4.48,1.1965208)
\psdots[linecolor=black, fillstyle=solid, dotstyle=o, dotsize=0.16, fillcolor=white](6.88,-2.8034792)
\rput[bl](10.74,-1.8834791){$T_j$}
\psdots[linecolor=black, dotsize=0.1](11.68,-3.6034791)
\psline[linecolor=black, linewidth=0.02, linestyle=dotted, dotsep=0.10583334cm](0.48,-2.8034792)(1.28,-2.0034792)(2.08,-1.2034792)(2.88,-2.0034792)
\psline[linecolor=black, linewidth=0.02, linestyle=dotted, dotsep=0.10583334cm](2.08,-1.2034792)(1.28,-0.40347916)(0.48,-1.2034792)
\psline[linecolor=black, linewidth=0.02, linestyle=dotted, dotsep=0.10583334cm](1.28,-0.40347916)(0.48,0.39652085)(1.28,1.1965208)(0.48,1.9965209)
\psline[linecolor=black, linewidth=0.02, linestyle=dotted, dotsep=0.10583334cm](2.08,1.9965209)(1.28,1.1965208)
\psline[linecolor=black, linewidth=0.02, linestyle=dotted, dotsep=0.10583334cm](2.88,1.1965208)(2.08,0.39652085)
\psline[linecolor=black, linewidth=0.02, linestyle=dotted, dotsep=0.10583334cm](2.08,0.39652085)(3.68,-1.2034792)
\psline[linecolor=black, linewidth=0.02, linestyle=dotted, dotsep=0.10583334cm](2.88,-0.40347916)(3.68,0.39652085)(2.88,1.1965208)
\psline[linecolor=black, linewidth=0.02, linestyle=dotted, dotsep=0.10583334cm](3.68,-1.2034792)(4.48,-0.40347916)
\psline[linecolor=black, linewidth=0.02, linestyle=dotted, dotsep=0.10583334cm](5.28,-1.2034792)(6.08,-0.40347916)(7.68,-2.0034792)
\psline[linecolor=black, linewidth=0.02, linestyle=dotted, dotsep=0.10583334cm](7.68,-0.40347916)(9.28,-2.0034792)
\psline[linecolor=black, linewidth=0.02, linestyle=dotted, dotsep=0.10583334cm](7.68,-2.0034792)(9.28,-0.40347916)
\psline[linecolor=black, linewidth=0.02, linestyle=dotted, dotsep=0.10583334cm](9.28,-2.0034792)(10.08,-1.2034792)(10.88,-2.0034792)
\psline[linecolor=black, linewidth=0.02, linestyle=dotted, dotsep=0.10583334cm](10.08,0.39652085)(7.68,2.796521)(5.28,0.39652085)(4.48,1.1965208)
\psline[linecolor=black, linewidth=0.02, linestyle=dotted, dotsep=0.10583334cm](5.28,1.9965209)(6.88,0.39652085)(8.48,1.9965209)
\psline[linecolor=black, linewidth=0.02, linestyle=dotted, dotsep=0.10583334cm](5.28,1.9965209)(6.08,2.796521)(7.68,1.1965208)
\psline[linecolor=black, linewidth=0.02, linestyle=dotted, dotsep=0.10583334cm](9.28,2.796521)(10.88,1.1965208)
\psline[linecolor=black, linewidth=0.02, linestyle=dotted, dotsep=0.10583334cm](10.08,1.9965209)(10.88,2.796521)
\psline[linecolor=black, linewidth=0.05](5.28,-2.8034792)(8.48,0.39652085)(10.08,-1.2034792)(10.88,-0.40347916)(11.68,-3.6034791)(5.28,-2.8034792)
\psdots[linecolor=red, fillstyle=solid, dotstyle=o, dotsize=0.16, fillcolor=white](6.08,-2.0034792)
\end{pspicture}
}
\caption{A mincut $\Sigma$ (red) cannot have two edges with a vertex on the $Ox$ axis}
	\label{Fig 6}
\end{figure}
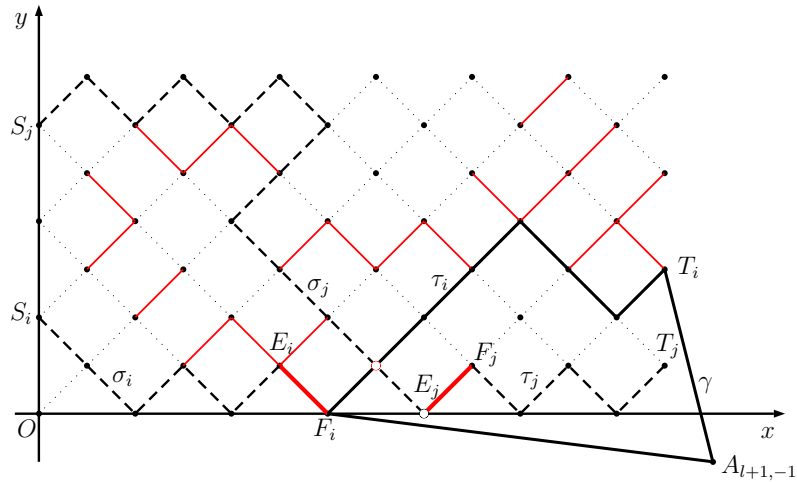

 We shall prove that any square with sides in $E_{l,w}^{(i)} $ has either two sides or none in $\Sigma $. Let \textit{MNPQ} be a square with sides in $E_{l,w}^{(i)} $ such that $MN=e_{i} \in \Sigma $. We know that one of the endpoints of $e_{i} $ (suppose $M=E_{i} $) is connected to a source node by an \textbf{X }-- path $\sigma _{i} $, and that the other one, $N=F_{i}$, is connected to one of the terminus nodes by an \textbf{X }-- path $\tau _{i} $. If all the other sides of the square were in $E_{l,w}^{(i)} -\Sigma $, then the \textbf{X }-- path $\sigma _{i} \cup MQ\cup QP\cup PN\cup \tau _{i} \subset E_{l,w}^{(i)} -\Sigma $ would connect a source node to a terminus node, so $\Sigma $ would not be a cutset. On the other hand, if \textit{MNPQ} has at least 3 edges in $\Sigma $, $e_{i} =MN,$ $e_{j} =MQ$ and $e_{k} =NP$, it follows that two opposite vertices are reachable from source nodes (suppose $M=E_{i} =E_{j} $, $P=E_{k} $) and the other two are reachable from terminus nodes ($N=F_{i} =F_{k} $, $Q=F_{j} $). If $\tau _{i} \cap \tau _{j} \ne \emptyset $, we denote by $\gamma $ the simple loop formed by $QN$, $\tau _{i} $ and  $\tau _{j} $. Otherwise, $\gamma =QN\cup \tau _{i} \cup T_{i} T_{j} \cup \tau _{j} $. One of the points \textit{M} and \textit{P} is in the interior region bounded by $\gamma $. Let \textit{M }be this point. Since \textit{M} is connected to the source node $S_{i} $ by $\sigma _{i} $ and $S_{i} $ is in the exterior of $\gamma $, it follows that $\sigma _{i} \cap \gamma \ne \emptyset $, so $\sigma _{i} \cap \tau _{_{j} } \ne \emptyset $, which is impossible (see Fig. \ref{Fig 7}). Thus, the square \textit{MNPQ} has exactly two sides in $\Sigma $. 

 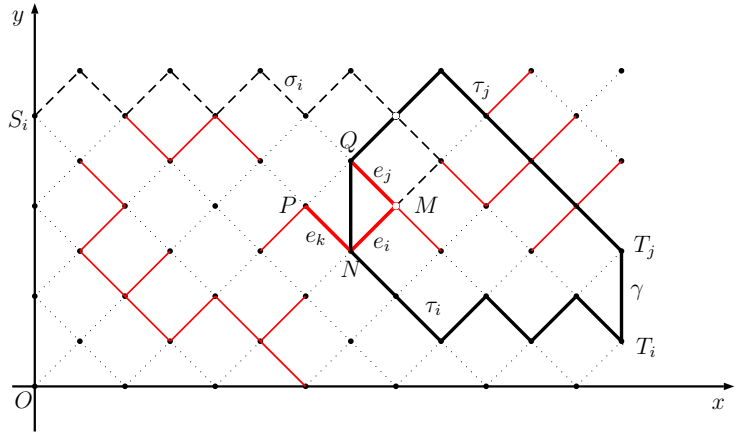
\begin{figure}[h]\centering
 \psscalebox{0.75} 
{
\begin{pspicture}(0,-3.753479)(12.786958,3.753479)
\psline[linecolor=black, linewidth=0.04, arrowsize=0.05291667cm 2.0,arrowlength=1.4,arrowinset=0.0]{->}(0.08,-2.9534793)(12.88,-2.9534793)
\psline[linecolor=black, linewidth=0.04, arrowsize=0.05291667cm 2.0,arrowlength=1.4,arrowinset=0.0]{->}(0.48,-3.7534792)(0.48,3.8465207)
\psdots[linecolor=black, dotsize=0.1](0.48,-2.9534793)
\psdots[linecolor=black, dotsize=0.1](2.08,-2.9534793)
\psdots[linecolor=black, dotsize=0.1](3.68,-2.9534793)
\psdots[linecolor=black, dotsize=0.1](5.28,-2.9534793)
\psdots[linecolor=black, dotsize=0.1](6.88,-2.9534793)
\psdots[linecolor=black, dotsize=0.1](8.48,-2.9534793)
\psdots[linecolor=black, dotsize=0.1](10.08,-2.9534793)
\psdots[linecolor=black, dotsize=0.1](1.28,-2.1534793)
\psdots[linecolor=black, dotsize=0.1](2.88,-2.1534793)
\psdots[linecolor=black, dotsize=0.1](4.48,-2.1534793)
\psdots[linecolor=black, dotsize=0.1](6.08,-2.1534793)
\psdots[linecolor=black, dotsize=0.1](7.68,-2.1534793)
\psdots[linecolor=black, dotsize=0.1](9.28,-2.1534793)
\psdots[linecolor=black, dotsize=0.1](8.48,-1.3534793)
\psdots[linecolor=black, dotsize=0.1](6.88,-1.3534793)
\psdots[linecolor=black, dotsize=0.1](5.28,-1.3534793)
\psdots[linecolor=black, dotsize=0.1](3.68,-1.3534793)
\psdots[linecolor=black, dotsize=0.1](2.08,-1.3534793)
\psdots[linecolor=black, dotsize=0.1](0.48,-1.3534793)
\psdots[linecolor=black, dotsize=0.1](7.68,-0.5534793)
\psdots[linecolor=black, dotsize=0.1](6.08,-0.5534793)
\psdots[linecolor=black, dotsize=0.1](4.48,-0.5534793)
\psdots[linecolor=black, dotsize=0.1](2.88,-0.5534793)
\psdots[linecolor=black, dotsize=0.1](1.28,-0.5534793)
\psdots[linecolor=black, dotsize=0.1](9.28,-0.5534793)
\psdots[linecolor=black, dotsize=0.1](9.28,1.0465207)
\psdots[linecolor=black, dotsize=0.1](9.28,2.6465206)
\psdots[linecolor=black, dotsize=0.1](8.48,0.2465207)
\psdots[linecolor=black, dotsize=0.1](6.88,0.2465207)
\psdots[linecolor=black, dotsize=0.1](7.68,1.0465207)
\psdots[linecolor=black, dotsize=0.1](8.48,1.8465207)
\psdots[linecolor=black, dotsize=0.1](0.48,0.2465207)
\psdots[linecolor=black, dotsize=0.1](0.48,1.8465207)
\psdots[linecolor=black, dotsize=0.1](1.28,2.6465206)
\psdots[linecolor=black, dotsize=0.1](1.28,1.0465207)
\psdots[linecolor=black, dotsize=0.1](2.08,1.8465207)
\psdots[linecolor=black, dotsize=0.1](2.88,2.6465206)
\psdots[linecolor=black, dotsize=0.1](2.08,0.2465207)
\psdots[linecolor=black, dotsize=0.1](2.88,1.0465207)
\psdots[linecolor=black, dotsize=0.1](3.68,1.8465207)
\psdots[linecolor=black, dotsize=0.1](4.48,2.6465206)
\psdots[linecolor=black, dotsize=0.1](3.68,0.2465207)
\psdots[linecolor=black, dotsize=0.1](4.48,1.0465207)
\psdots[linecolor=black, dotsize=0.1](5.28,1.8465207)
\psdots[linecolor=black, dotsize=0.1](6.08,2.6465206)
\psdots[linecolor=black, dotsize=0.1](5.28,0.2465207)
\psdots[linecolor=black, dotsize=0.1](6.08,1.0465207)
\psdots[linecolor=black, dotsize=0.1](6.88,1.8465207)
\psdots[linecolor=black, dotsize=0.1](7.68,2.6465206)
\psdots[linecolor=black, dotsize=0.1](10.08,-1.3534793)
\psdots[linecolor=black, dotsize=0.1](10.08,0.2465207)
\psdots[linecolor=black, dotsize=0.1](10.08,1.8465207)
\psdots[linecolor=black, dotsize=0.1](10.88,-2.1534793)
\psdots[linecolor=black, dotsize=0.1](10.88,-0.5534793)
\psdots[linecolor=black, dotsize=0.1](10.88,1.0465207)
\psdots[linecolor=black, dotsize=0.1](10.88,2.6465206)
\psline[linecolor=red, linewidth=0.032](4.48,-2.1534793)(5.28,-2.9534793)
\rput[bl](0.08,3.4465208){$y$}
\rput[bl](0.12,-3.3334794){$O$}
\rput[bl](12.48,-3.3534794){$x$}
\rput[bl](0.0,1.5865207){$S_i$}
\rput[bl](11.04,-1.4134793){$\gamma$}
\rput[bl](7.4,-1.6934793){$\tau_i$}
\rput[bl](4.9,2.3065207){$\sigma_i$}
\rput[bl](8.24,2.1465206){$\tau_j$}
\rput[bl](11.1,-0.6934793){$T_j$}
\psline[linecolor=red, linewidth=0.03](4.48,-2.1534793)(3.68,-1.3534793)
\psline[linecolor=red, linewidth=0.03](4.48,-2.1534793)(5.28,-1.3534793)
\psline[linecolor=red, linewidth=0.03](3.68,-1.3534793)(2.88,-2.1534793)
\psline[linecolor=red, linewidth=0.03](1.28,-0.5534793)(2.08,0.2465207)
\psline[linecolor=red, linewidth=0.03](1.28,1.0465207)(2.08,0.2465207)
\psline[linecolor=red, linewidth=0.03](2.08,1.8465207)(2.88,1.0465207)(3.68,1.8465207)
\psline[linecolor=red, linewidth=0.03](4.48,-0.5534793)(5.28,0.2465207)(6.08,-0.5534793)(6.88,0.2465207)(7.68,-0.5534793)
\psline[linecolor=red, linewidth=0.03](7.68,1.0465207)(8.48,0.2465207)(10.08,1.8465207)
\psline[linecolor=red, linewidth=0.03](9.28,-0.5534793)(10.88,1.0465207)
\psline[linecolor=red, linewidth=0.03](8.48,1.8465207)(9.28,2.6465206)
\psline[linecolor=red, linewidth=0.03](3.68,1.8465207)(4.48,1.0465207)
\rput[bl](11.14,-2.4334793){$T_i$}
\psline[linecolor=red, linewidth=0.03](6.88,0.2465207)(6.08,1.0465207)
\psline[linecolor=black, linewidth=0.03, linestyle=dashed, dash=0.17638889cm 0.10583334cm](6.88,0.2465207)(7.68,1.0465207)(6.08,2.6465206)(5.28,1.8465207)(4.48,2.6465206)(3.68,1.8465207)(2.88,2.6465206)(2.08,1.8465207)(1.28,2.6465206)(0.48,1.8465207)
\psline[linecolor=black, linewidth=0.03, linestyle=dashed, dash=0.17638889cm 0.10583334cm](6.08,-0.5534793)(7.68,-2.1534793)(8.48,-1.3534793)(9.28,-2.1534793)(10.08,-1.3534793)(10.88,-2.1534793)
\psline[linecolor=black, linewidth=0.03, linestyle=dashed, dash=0.17638889cm 0.10583334cm](6.08,1.0465207)(7.68,2.6465206)(10.88,-0.5534793)
\psline[linecolor=red, linewidth=0.03](1.28,-0.5534793)(2.88,-2.1534793)
\psline[linecolor=red, linewidth=0.064](5.28,0.2465207)(6.08,-0.5534793)(6.88,0.2465207)(6.08,1.0465207)
\psdots[linecolor=red, dotstyle=o, dotsize=0.14, fillcolor=white](6.88,0.2465207)
\rput[bl](7.2,0.1265207){$M$}
\rput[bl](5.9,-1.0134794){$N$}
\rput[bl](4.8,0.1265207){$P$}
\rput[bl](5.88,1.1865207){$Q$}
\rput[bl](6.48,-0.5534793){$e_i$}
\rput[bl](6.48,0.6465207){$e_j$}
\rput[bl](5.28,-0.45347932){$e_k$}
\psline[linecolor=red, linewidth=0.03](2.08,-1.3534793)(2.88,-0.5534793)
\psline[linecolor=black, linewidth=0.06](6.08,1.0465207)(6.08,-0.5534793)(7.68,-2.1534793)(8.48,-1.3534793)(9.28,-2.1534793)(10.08,-1.3534793)(10.88,-2.1534793)(10.88,-0.5534793)(7.68,2.6465206)(6.08,1.0465207)
\psdots[linecolor=black, dotstyle=o, dotsize=0.14, fillcolor=white](6.88,1.8465207)
\psline[linecolor=black, linewidth=0.02, linestyle=dotted, dotsep=0.10583334cm](0.48,-2.9534793)(1.28,-2.1534793)(2.08,-2.9534793)(2.88,-2.1534793)(3.68,-2.9534793)(4.48,-2.1534793)
\psline[linecolor=black, linewidth=0.02, linestyle=dotted, dotsep=0.10583334cm](0.48,1.8465207)(1.28,1.0465207)(0.48,0.2465207)(1.28,-0.5534793)(0.48,-1.3534793)(1.28,-2.1534793)(2.08,-1.3534793)
\psline[linecolor=black, linewidth=0.02, linestyle=dotted, dotsep=0.10583334cm](2.08,1.8465207)(1.28,1.0465207)
\psline[linecolor=black, linewidth=0.02, linestyle=dotted, dotsep=0.10583334cm](2.88,1.0465207)(2.08,0.2465207)(3.68,-1.3534793)(4.48,-0.5534793)(2.88,1.0465207)
\psline[linecolor=black, linewidth=0.02, linestyle=dotted, dotsep=0.10583334cm](3.68,0.2465207)(5.28,1.8465207)(6.08,1.0465207)(5.28,0.2465207)(4.48,1.0465207)
\psline[linecolor=black, linewidth=0.02, linestyle=dotted, dotsep=0.10583334cm](3.68,0.2465207)(2.88,-0.5534793)
\psline[linecolor=black, linewidth=0.02, linestyle=dotted, dotsep=0.10583334cm](4.48,-0.5534793)(6.88,-2.9534793)
\psline[linecolor=black, linewidth=0.02, linestyle=dotted, dotsep=0.10583334cm](5.28,-1.3534793)(6.08,-0.5534793)
\psline[linecolor=black, linewidth=0.02, linestyle=dotted, dotsep=0.10583334cm](5.28,-2.9534793)(8.48,0.2465207)(10.08,-1.3534793)
\psline[linecolor=black, linewidth=0.02, linestyle=dotted, dotsep=0.10583334cm](6.88,-2.9534793)(7.68,-2.1534793)(8.48,-2.9534793)(9.28,-2.1534793)(10.08,-2.9534793)(10.88,-2.1534793)
\psline[linecolor=black, linewidth=0.02, linestyle=dotted, dotsep=0.10583334cm](7.68,-0.5534793)(8.48,-1.3534793)(9.28,-0.5534793)
\psline[linecolor=black, linewidth=0.02, linestyle=dotted, dotsep=0.10583334cm](10.88,-0.5534793)(10.08,-1.3534793)
\psline[linecolor=black, linewidth=0.02, linestyle=dotted, dotsep=0.10583334cm](7.68,1.0465207)(8.48,1.8465207)
\psline[linecolor=black, linewidth=0.02, linestyle=dotted, dotsep=0.10583334cm](9.28,2.6465206)(10.88,1.0465207)
\psline[linecolor=black, linewidth=0.02, linestyle=dotted, dotsep=0.10583334cm](10.08,1.8465207)(10.88,2.6465206)
\end{pspicture}
}
\caption{A square of the network  cannot have three sides in $\Sigma$ }
	\label{Fig 7}
\end{figure}
We can change the indices of the edges of $\Sigma $ such that $e_{1} =e_{I} $, $e_{2} $ is the other edge in $\Sigma $ of the square with a side $e_{1} $, and so on up to the final edge, $e_{n} =e_{F} $. It follows that $\overline{\Sigma }=\left\{\bar{e}_{1} ,\bar{e}_{2} ,\ldots ,\bar{e}_{n} \right\}$ is an \textbf{X }-- path connecting a point on \textit{Ox} to a point on the horizontal line $y=w$, so $\overline{\Sigma }$ is a minpath in $\overline{H_{l,w}^{(i)} }$ (see Fig. \ref{Fig 8}).

 \begin{figure}[h]\centering
 \psscalebox{0.75} 
{
\begin{pspicture}(0,-3.753479)(12.706958,3.753479)
\psline[linecolor=black, linewidth=0.04, arrowsize=0.05291667cm 2.0,arrowlength=1.4,arrowinset=0.0]{->}(0.0,-2.9534793)(12.8,-2.9534793)
\psline[linecolor=black, linewidth=0.04, arrowsize=0.05291667cm 2.0,arrowlength=1.4,arrowinset=0.0]{->}(0.4,-3.7534792)(0.4,3.8465207)
\psdots[linecolor=black, dotsize=0.1](0.4,-2.9534793)
\psdots[linecolor=black, dotsize=0.1](2.0,-2.9534793)
\psdots[linecolor=black, dotsize=0.1](3.6,-2.9534793)
\psdots[linecolor=black, dotsize=0.1](5.2,-2.9534793)
\psdots[linecolor=black, dotsize=0.1](6.8,-2.9534793)
\psdots[linecolor=black, dotsize=0.1](8.4,-2.9534793)
\psdots[linecolor=black, dotsize=0.1](10.0,-2.9534793)
\psdots[linecolor=black, dotsize=0.1](1.2,-2.1534793)
\psdots[linecolor=black, dotsize=0.1](2.8,-2.1534793)
\psdots[linecolor=black, dotsize=0.1](4.4,-2.1534793)
\psdots[linecolor=black, dotsize=0.1](6.0,-2.1534793)
\psdots[linecolor=black, dotsize=0.1](7.6,-2.1534793)
\psdots[linecolor=black, dotsize=0.1](9.2,-2.1534793)
\psdots[linecolor=black, dotsize=0.1](8.4,-1.3534793)
\psdots[linecolor=black, dotsize=0.1](6.8,-1.3534793)
\psdots[linecolor=black, dotsize=0.1](5.2,-1.3534793)
\psdots[linecolor=black, dotsize=0.1](3.6,-1.3534793)
\psdots[linecolor=black, dotsize=0.1](2.0,-1.3534793)
\psdots[linecolor=black, dotsize=0.1](0.4,-1.3534793)
\psdots[linecolor=black, dotsize=0.1](7.6,-0.5534793)
\psdots[linecolor=black, dotsize=0.1](6.0,-0.5534793)
\psdots[linecolor=black, dotsize=0.1](4.4,-0.5534793)
\psdots[linecolor=black, dotsize=0.1](2.8,-0.5534793)
\psdots[linecolor=black, dotsize=0.1](1.2,-0.5534793)
\psdots[linecolor=black, dotsize=0.1](9.2,-0.5534793)
\psdots[linecolor=black, dotsize=0.1](9.2,1.0465207)
\psdots[linecolor=black, dotsize=0.1](9.2,2.6465206)
\psdots[linecolor=black, dotsize=0.1](8.4,0.2465207)
\psdots[linecolor=black, dotsize=0.1](6.8,0.2465207)
\psdots[linecolor=black, dotsize=0.1](7.6,1.0465207)
\psdots[linecolor=black, dotsize=0.1](8.4,1.8465207)
\psdots[linecolor=black, dotsize=0.1](0.4,0.2465207)
\psdots[linecolor=black, dotsize=0.1](0.4,1.8465207)
\psdots[linecolor=black, dotsize=0.1](1.2,2.6465206)
\psdots[linecolor=black, dotsize=0.1](1.2,1.0465207)
\psdots[linecolor=black, dotsize=0.1](2.0,1.8465207)
\psdots[linecolor=black, dotsize=0.1](2.8,2.6465206)
\psdots[linecolor=black, dotsize=0.1](2.0,0.2465207)
\psdots[linecolor=black, dotsize=0.1](2.8,1.0465207)
\psdots[linecolor=black, dotsize=0.1](3.6,1.8465207)
\psdots[linecolor=black, dotsize=0.1](4.4,2.6465206)
\psdots[linecolor=black, dotsize=0.1](3.6,0.2465207)
\psdots[linecolor=black, dotsize=0.1](4.4,1.0465207)
\psdots[linecolor=black, dotsize=0.1](5.2,1.8465207)
\psdots[linecolor=black, dotsize=0.1](6.0,2.6465206)
\psdots[linecolor=black, dotsize=0.1](5.2,0.2465207)
\psdots[linecolor=black, dotsize=0.1](6.0,1.0465207)
\psdots[linecolor=black, dotsize=0.1](7.6,2.6465206)
\psdots[linecolor=black, dotsize=0.1](10.0,-1.3534793)
\psdots[linecolor=black, dotsize=0.1](10.0,0.2465207)
\psdots[linecolor=black, dotsize=0.1](10.0,1.8465207)
\psdots[linecolor=black, dotsize=0.1](10.8,-2.1534793)
\psdots[linecolor=black, dotsize=0.1](10.8,-0.5534793)
\psdots[linecolor=black, dotsize=0.1](10.8,1.0465207)
\psdots[linecolor=black, dotsize=0.1](10.8,2.6465206)
\psline[linecolor=red, linewidth=0.032](4.4,-2.1534793)(5.2,-2.9534793)
\rput[bl](0.0,3.4465208){$y$}
\rput[bl](0.04,-3.3334794){$O$}
\rput[bl](12.4,-3.3534794){$x$}
\psline[linecolor=red, linewidth=0.03](4.4,-2.1534793)(3.6,-1.3534793)
\psline[linecolor=red, linewidth=0.03](4.4,-2.1534793)(5.2,-1.3534793)
\psline[linecolor=red, linewidth=0.03](3.6,-1.3534793)(2.8,-2.1534793)
\psline[linecolor=red, linewidth=0.03](1.2,-0.5534793)(2.0,0.2465207)
\psline[linecolor=red, linewidth=0.03](1.2,1.0465207)(2.0,0.2465207)
\psline[linecolor=red, linewidth=0.03](2.0,1.8465207)(2.8,1.0465207)(3.6,1.8465207)
\psline[linecolor=red, linewidth=0.03](4.4,-0.5534793)(5.2,0.2465207)(6.0,-0.5534793)(6.8,0.2465207)(7.6,-0.5534793)
\psline[linecolor=red, linewidth=0.03](7.6,1.0465207)(8.4,0.2465207)(10.0,1.8465207)
\psline[linecolor=red, linewidth=0.03](9.2,-0.5534793)(10.8,1.0465207)
\psline[linecolor=red, linewidth=0.03](8.4,1.8465207)(9.2,2.6465206)
\psline[linecolor=red, linewidth=0.03](3.6,1.8465207)(4.4,1.0465207)
\rput[bl](4.98,-2.6734793){$e_1$}
\rput[bl](3.84,-1.6134793){$e_3$}
\rput[bl](4.3,-1.9734793){$e_2$}
\psline[linecolor=red, linewidth=0.03](2.0,-1.3534793)(2.8,-0.5534793)
\psline[linecolor=black, linewidth=0.02, linestyle=dotted, dotsep=0.10583334cm](0.4,-2.9534793)(1.2,-2.1534793)(2.0,-2.9534793)(2.8,-2.1534793)(3.6,-2.9534793)(4.4,-2.1534793)
\psline[linecolor=black, linewidth=0.02, linestyle=dotted, dotsep=0.10583334cm](0.4,1.8465207)(1.2,1.0465207)(0.4,0.2465207)(1.2,-0.5534793)(0.4,-1.3534793)(1.2,-2.1534793)(2.0,-1.3534793)
\psline[linecolor=black, linewidth=0.02, linestyle=dotted, dotsep=0.10583334cm](2.0,1.8465207)(1.2,1.0465207)
\psline[linecolor=black, linewidth=0.02, linestyle=dotted, dotsep=0.10583334cm](2.8,1.0465207)(2.0,0.2465207)(3.6,-1.3534793)(4.4,-0.5534793)(2.8,1.0465207)
\psline[linecolor=black, linewidth=0.02, linestyle=dotted, dotsep=0.10583334cm](3.6,0.2465207)(2.8,-0.5534793)
\psline[linecolor=black, linewidth=0.02, linestyle=dotted, dotsep=0.10583334cm](4.4,-0.5534793)(6.8,-2.9534793)
\psline[linecolor=black, linewidth=0.02, linestyle=dotted, dotsep=0.10583334cm](5.2,-1.3534793)(6.0,-0.5534793)
\psline[linecolor=black, linewidth=0.02, linestyle=dotted, dotsep=0.10583334cm](5.2,-2.9534793)(8.4,0.2465207)(10.0,-1.3534793)
\psline[linecolor=black, linewidth=0.02, linestyle=dotted, dotsep=0.10583334cm](6.8,-2.9534793)(7.6,-2.1534793)(8.4,-2.9534793)(9.2,-2.1534793)(10.0,-2.9534793)(10.8,-2.1534793)
\psline[linecolor=black, linewidth=0.02, linestyle=dotted, dotsep=0.10583334cm](7.6,-0.5534793)(8.4,-1.3534793)(9.2,-0.5534793)
\psline[linecolor=black, linewidth=0.02, linestyle=dotted, dotsep=0.10583334cm](10.8,-0.5534793)(10.0,-1.3534793)
\psline[linecolor=black, linewidth=0.02, linestyle=dotted, dotsep=0.10583334cm](7.6,1.0465207)(8.4,1.8465207)
\psline[linecolor=black, linewidth=0.02, linestyle=dotted, dotsep=0.10583334cm](9.2,2.6465206)(10.8,1.0465207)
\psline[linecolor=black, linewidth=0.02, linestyle=dotted, dotsep=0.10583334cm](10.0,1.8465207)(10.8,2.6465206)
\psline[linecolor=red, linewidth=0.03](4.4,1.0465207)(3.6,0.2465207)
\psline[linecolor=red, linewidth=0.03](10.0,0.2465207)(10.8,-0.5534793)
\psdots[linecolor=black, dotsize=0.1](6.8,1.8465207)
\psline[linecolor=black, linewidth=0.02, linestyle=dotted, dotsep=0.10583334cm](1.2,-0.5534793)(2.8,-2.1534793)
\psline[linecolor=black, linewidth=0.02, linestyle=dotted, dotsep=0.10583334cm](1.2,2.6465206)(2.0,1.8465207)
\psline[linecolor=black, linewidth=0.02, linestyle=dotted, dotsep=0.10583334cm](0.4,1.8465207)(1.2,2.6465206)
\psline[linecolor=black, linewidth=0.02, linestyle=dotted, dotsep=0.10583334cm](2.0,1.8465207)(2.8,2.6465206)(3.6,1.8465207)(4.4,2.6465206)(6.8,0.2465207)
\psline[linecolor=black, linewidth=0.02, linestyle=dotted, dotsep=0.10583334cm](5.2,0.2465207)(7.6,2.6465206)(10.0,0.2465207)
\psline[linecolor=black, linewidth=0.02, linestyle=dotted, dotsep=0.10583334cm](5.2,0.2465207)(4.4,1.0465207)(6.0,2.6465206)(7.6,1.0465207)(6.8,0.2465207)
\psline[linecolor=black, linewidth=0.02, linestyle=dotted, dotsep=0.10583334cm](6.0,-0.5534793)(7.6,-2.1534793)(8.4,-1.3534793)(9.2,-2.1534793)(10.0,-1.3534793)(10.8,-2.1534793)
\psline[linecolor=black, linewidth=0.04](4.4,-2.9534793)(5.2,-2.1534793)(4.4,-1.3534793)(3.6,-2.1534793)(2.8,-1.3534793)(2.0,-0.5534793)(1.2,0.2465207)(2.0,1.0465207)(2.8,1.8465207)(3.6,1.0465207)(4.4,0.2465207)(5.2,-0.5534793)(6.0,0.2465207)(6.8,-0.5534793)(7.6,0.2465207)(8.4,1.0465207)(9.2,0.2465207)(10.0,-0.5534793)(10.8,0.2465207)(10.0,1.0465207)(9.2,1.8465207)(8.4,2.6465206)
\psdots[linecolor=black, dotstyle=o, dotsize=0.14, fillcolor=white](4.4,-2.9534793)
\psdots[linecolor=black, dotstyle=o, dotsize=0.14, fillcolor=white](5.2,-2.1534793)
\psdots[linecolor=black, dotstyle=o, dotsize=0.14, fillcolor=white](4.4,-1.3534793)
\psdots[linecolor=black, dotstyle=o, dotsize=0.14, fillcolor=white](3.6,-2.1534793)
\psdots[linecolor=black, dotstyle=o, dotsize=0.14, fillcolor=white](2.8,-1.3534793)
\psdots[linecolor=black, dotstyle=o, dotsize=0.14, fillcolor=white](2.0,-0.5534793)
\psdots[linecolor=black, dotstyle=o, dotsize=0.14, fillcolor=white](1.2,0.2465207)
\psdots[linecolor=black, dotstyle=o, dotsize=0.14, fillcolor=white](2.0,1.0465207)
\psdots[linecolor=black, dotstyle=o, dotsize=0.14, fillcolor=white](2.8,1.8465207)
\psdots[linecolor=black, dotstyle=o, dotsize=0.14, fillcolor=white](3.6,1.0465207)
\psdots[linecolor=black, dotstyle=o, dotsize=0.14, fillcolor=white](4.4,0.2465207)
\psdots[linecolor=black, dotstyle=o, dotsize=0.14, fillcolor=white](5.2,-0.5534793)
\psdots[linecolor=black, dotstyle=o, dotsize=0.14, fillcolor=white](6.0,0.2465207)
\psdots[linecolor=black, dotstyle=o, dotsize=0.14, fillcolor=white](6.8,-0.5534793)
\psdots[linecolor=black, dotstyle=o, dotsize=0.14, fillcolor=white](7.6,0.2465207)
\psdots[linecolor=black, dotstyle=o, dotsize=0.14, fillcolor=white](8.4,1.0465207)
\psdots[linecolor=black, dotstyle=o, dotsize=0.14, fillcolor=white](9.2,0.2465207)
\psdots[linecolor=black, dotstyle=o, dotsize=0.14, fillcolor=white](10.0,-0.5534793)
\psdots[linecolor=black, dotstyle=o, dotsize=0.14, fillcolor=white](10.8,0.2465207)
\psdots[linecolor=black, dotstyle=o, dotsize=0.14, fillcolor=white](10.0,1.0465207)
\psdots[linecolor=black, dotstyle=o, dotsize=0.14, fillcolor=white](9.2,1.8465207)
\psdots[linecolor=black, dotstyle=o, dotsize=0.14, fillcolor=white](8.4,2.6465206)
\rput[bl](8.66,2.4865208){$e_n$}
\rput[bl](9.78,1.3265207){$e_{n-1}$}
\end{pspicture}
}
\caption{If $\Sigma$ is a mincut in $H_{l,w}^{(i)} $, then $\overline{\Sigma}$ is a minpath in $\overline{H_{l,w}^{(i)} }$}
	\label{Fig 8}
\end{figure}
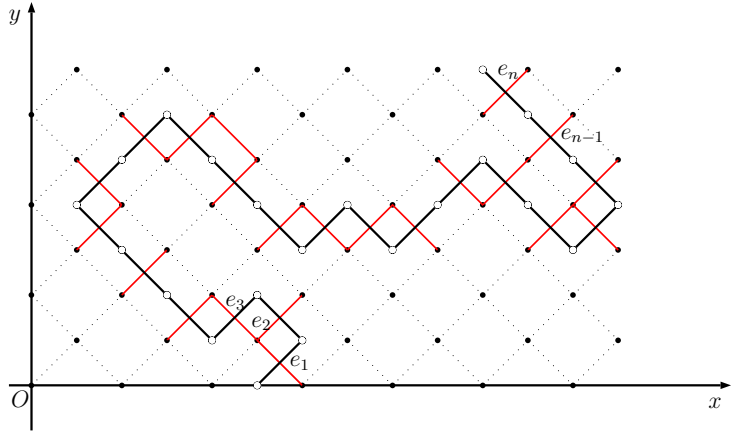

ii) Since $\Sigma $ is a minpath in $H_{l,w}^{(i)} $, it follows that it is an \textbf{X }-- path connecting a source node $S_{i} $ (located on \textit{Oy})  to a terminus node $T_{j} $ (located on the straight line $x=l$). We consider the simple closed curve $\gamma =\Sigma \cup T_{j} A_{l,-1} \cup A_{l,-1} A_{0,-1} \cup A_{0,-1} S_{i} $. If $\sigma $ is an X\textbf{ }-- path in $\overline{H_{l,w}^{(i)} }$, connecting a source node $S'_{i} =A_{x,0} $ (an interior point of $\gamma $) to a target node $T'_{j} =A_{z,w} $ (an exterior point) then, by Jordan Curve Theorem, it follows that $\sigma \cap \gamma \ne \emptyset$. Hence $\sigma $ contains at least one edge that ``cuts'' an edge of $\Sigma $ (an edge of $\overline{\Sigma }$). Thus, any pathset in $\overline{H_{l,w}^{(i)} }$ contains at least one edge of $\overline{\Sigma }$, so $\overline{\Sigma }$ is a cutset. It remains to prove that $\overline{\Sigma }$ is a mincut. Suppose that some edges can be eliminated from $\overline{\Sigma }$ to obtain a mincut $\Sigma '\subset \overline{\Sigma }$. As it was proved above, it follows that $\overline{\Sigma '}\subset \Sigma $ is a minpath in $H_{l,w}^{(i)} $, which is impossible, because $\Sigma $ is a minpath.         
\end{proof}

Theorem \ref{t1} states that $\Sigma \subset E_{l,w}^{(i)} $ is a minpath in $H_{l,w}^{(i)} $ if and only if $\overline{\Sigma }\subset \overline{E_{l,w}^{(i)} }$ is a mincut in $\overline{H_{l,w}^{(i)} }$. The symmetric statement is also true, by Remark \ref{r2}:  $\Sigma $ is a \textit{mincut }in $H_{l,w}^{(i)} $ if and only if $\overline{\Sigma }$ is a \textit{minpath }in $\overline{H_{l,w}^{(i)} }$. The corollary below gives a more general result, for any pathset and, respectively, cutset.

\begin{corollary}\label{c1} Let $\Sigma =\left\{e_{1} ,e_{2} ,\ldots ,e_{n} \right\}\subset E_{l,w}^{(i)} $ be a subset of edges of the network $H_{l,w}^{(i)} $ and let $\overline{\Sigma }=\left\{\bar{e}_{1} ,\bar{e}_{2} ,\ldots ,\bar{e}_{n} \right\}\subset \overline{E_{l,w}^{(i)} }$ be the set of complementary edges. Then $\Sigma $ is a \textit{pathset} in $H_{l,w}^{(i)} $ if and only if $\overline{\Sigma }$ is a \textit{cutset} in $\overline{H_{l,w}^{(i)} }$.
\end{corollary}

As a consequence, by using the equation \eqref{3} and Remark \ref{r3}, we have the following corollary:

\begin{corollary}\label{c2} For any $l,w\ge 1$ and $i=1,2$ the following relation is true for all $p\in [0,1]$:
\begin{equation} \label{5}
 h_{l,w}^{(i)} (p)=1-h_{w,l}^{(2/i)} (1-p). \end{equation}
\end{corollary}

By Remark \ref{r1}, if at least one of $l$ or $w$ is an odd number, then $h_{l,w}^{(1)} =h_{l,w}^{(2)} =h_{l,w} $. Consequently:

\begin{corollary} If at least one of the integers \textit{l } and \textit{w } is odd, then the following relation is true for all $p\in [0,1]$:
\begin{equation} \label{6} h_{l,w} (p)=1-h_{w,l} (1-p) .\end{equation} 
\end{corollary}

For $l\ne w$ this means that the plots of the reliability polynomials $h_{l,w} (p)$ and $h_{w,l} (p)$ are symmetric one to each other with respect to the point $\left({\frac{1}{2}} ,{\frac{1}{2}} \right)$ (see Fig. \ref{Fig 9}).

 \begin{figure}[h]\centering
\scalebox{0.75}
{
\begin{tikzpicture}
\begin{axis}[
    axis lines = left,
    xlabel = $p$,
    ylabel = $h(p)$,
		xtick={0,0.5,1},
    ytick={0,0.5,1},
		legend pos=south east,
]
\addplot [
		domain=0:1, 
    samples=100, 
    color=red,
		ultra thick
]
{5*x^2 - 4*x^3-3*x^4+4*x^5 - x^6};
\addlegendentry{$h_{2,3}(p)$}
\addplot [
    domain=0:1, 
    samples=100, 
    color=blue,
		ultra thick
    ]
    {4*x^3-2*x^4-2*x^5 + x^6};
\addlegendentry{$h_{3,2}(p)$}
 \addplot[
    color=black,
    mark=x,
    ]
    coordinates {
    (0.5,0.5)
    };
		\draw [dashed,ultra thin] (0,0.5)-- (0.5,0.5);
		\draw [dashed,ultra thin] (0.5,0.5)-- (0.5,0);
				\draw (0,0)--(1,1);
\end{axis}
\end{tikzpicture}
}
\caption{The plots of $h_{l,w} (p)$ and $h_{w,l} (p)$ when at least one dimension is odd.}
\label{Fig 9}
\end{figure}
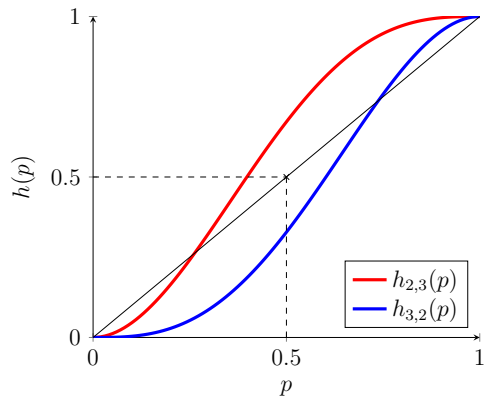

For $l=w=2k+1$ it means that the point $\left({\tfrac{1}{2}} ,{\tfrac{1}{2}} \right)$ is a center of symmetry for the plot of the reliability polynomial $h_{2k+1,2k+1} (p)$ (see Fig. \ref{Fig 10}).

 \begin{figure}[h]\centering
\scalebox{0.75}
{
\begin{tikzpicture}
\begin{axis}[
    axis lines = left,
    xlabel = $p$,
    ylabel = $h(p)$,
		xtick={0,0.5,1},
    ytick={0,0.5,1},
		legend pos=south east,
]
\addplot [
		domain=0:1, 
    samples=100, 
    color=red,
		ultra thick
]
{8*x^3-6*x^4-6*x^5+12*x^7-9*x^8+2*x^9};
\addlegendentry{$h_{3,3}(p)$}

 \addplot[
    color=black,
    mark=x,
    ]
    coordinates {
    (0.5,0.5)
    };
		\draw [dashed,ultra thin] (0,0.5)-- (0.5,0.5);
		\draw [dashed,ultra thin] (0.5,0.5)-- (0.5,0);
		\draw (0,0)--(1,1);
\end{axis}

\end{tikzpicture}
}
\caption{The plot of $h_{2k+1,2k+1}(p)$}
	\label{Fig 10}
\end{figure}
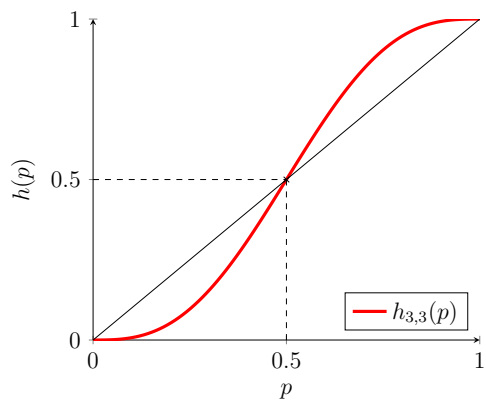

\begin{corollary}\label{c4} Let $h(p)=h_{l,w}^{(i)} (p)$ be the reliability polynomial of a hammock network of dimensions $(l,w)$, either of kind 1 or 2. Then the derivatives of $h$ verify the following relations:
\begin{equation} \label{7} h^{(k)} (0)=0,  \forall k=0,1,\ldots ,l-1 \end{equation} 
\begin{equation} \label{8} h(1)=1, \; h^{(k)} (1)=0,  \forall k=1,2,\ldots ,w-1 \end{equation} 
\end{corollary}
\begin{proof}
Since any pathset of a hammock network has at least \textit{l} edges, by equation \eqref{1}, we have:
\begin{equation} \label{9} h(p)=\sum _{i=l}^{wl}N_{i} p^{i} (1-p)^{wl-i}  =\sum _{i=l}^{wl}b_{i} p^{i}   \end{equation} 
and relation \eqref{7}  follows immediately. 

Let $\overline{h}(p)=h_{w,l}^{(2/i)} (p)$ be the reliability polynomial of the dual network, Since \textit{w} is the length of the dual network, it follows by \eqref{7} that $\overline{h}^{(k)}(0)=0,\; \forall k=0,1,\ldots ,w-1$.  By Corollary \ref{c2} we have that $h(p)=1-\overline{h}(1-p)$, and it follows that $h(1)=1$ and  $h^{(k)} (p)=(-1)^{k+1} \overline{h}^{(k)} (1-p)$ for all $k\ge 1$. For $p=1$ we obtain $h^{(k)} (1)=(-1)^{k+1} \overline{h}^{(k)}(0)=0,\; \forall k=1,2,\ldots ,w-1$.
\end{proof}

\section*{Acknowledgments} This research was funded by the European Union through the  European Regional Development Fund under the Competitiveness Operational Program (BioCell-NanoART = Novel Bio-inspired Cellular Nano-architectures, POC-A1-A1.1.4-E nr. 30/2016).





\end{document}